\documentclass[12pt]{amsart}
\usepackage{amsmath,amsthm,amssymb}
\usepackage{mathtools, etoolbox}
\usepackage{dsfont}
\usepackage{mathrsfs} %use \mathscr
\usepackage[all,cmtip]{xy}
\usepackage{patchcmd}
\input{diagxy}
\usepackage[backref=page]{hyperref}
\usepackage{bm}
\usepackage{enumerate}
\usepackage{pbox}
\usepackage[usenames,dvipsnames,svgnames,table]{xcolor}
\usepackage[normalem]{ulem}
\usepackage[top=2.2cm, bottom=2.2cm, left=2cm, right=2cm]{geometry}
\usepackage{tikz}
\usetikzlibrary{arrows,positioning}
\allowdisplaybreaks

\newcommand{\memph}[1]{{\color{red}\emph{#1}}}

\makeatletter
\patchcommand\@starttoc{\begin{quote}}{\end{quote}}
\def\@tocline#1#2#3#4#5#6#7{\relax
  \ifnum #1>\c@tocdepth % then omit
  \else
    \par \addpenalty\@secpenalty\addvspace{#2}%
    \begingroup \hyphenpenalty\@M
    \@ifempty{#4}{%
      \@tempdima\csname r@tocindent\number#1\endcsname\relax
    }{%
      \@tempdima#4\relax
    }%
    \parindent\z@ \leftskip#3\relax \advance\leftskip\@tempdima\relax
    \rightskip\@pnumwidth plus4em \parfillskip-\@pnumwidth
    #5\leavevmode\hskip-\@tempdima
      \ifcase #1
       \or\or \hskip 1em \or \hskip 2em \else \hskip 3em \fi%
      #6\nobreak\relax
    \dotfill\hbox to\@pnumwidth{\@tocpagenum{#7}}\par
    \nobreak
    \endgroup
  \fi}
\makeatother

 \theoremstyle{plain}
 \newtheorem{thm}{Theorem}[section]
 \newtheorem{cor}[thm]{Corollary}
 \newtheorem{lem}[thm]{Lemma}
 
 \newtheorem{prop}[thm]{Proposition}

\theoremstyle{definition}
 \newtheorem{defn}[thm]{Definition}

\theoremstyle{remark}
 \newtheorem{rem}[thm]{Remark}

 \newtheorem{nota}[thm]{Notation}
 \newtheorem{conv}[thm]{Convention}
 \newtheorem{exam}[thm]{Example}

 \numberwithin{equation}{section}

\theoremstyle{plain}

\DeclareMathOperator{\UU}{\mathcal{U}}

 \DeclareMathOperator{\ran}{ran}
 \DeclareMathOperator{\dom}{dom}

 \DeclareMathOperator{\id}{id}

 \DeclareMathOperator{\supp}{supp}

 \DeclareMathOperator{\dcl}{dcl}
 \DeclareMathOperator{\pr}{pr}

\def\XXint#1#2#3{{\setbox0=\hbox{$#1{#2#3}{\int}$}
\vcenter{\hbox{$#2#3$}}\kern-.5\wd0}}

\newcommand{\Z}{\mathds{Z}}

\newcommand{\Q}{\mathds{Q}}

\newcommand{\N}{\mathds{N}}

\newcommand{\R}{\mathds{R}}

\newcommand{\dO}{\mathds{O}}

\newcommand{\omin}{$o$\nobreakdash}

\newcommand{\T}{$T$\nobreakdash}
\newcommand{\ga}{\mathfrak{a}}

\newcommand{\gp}{\mathfrak{p}}
\newcommand{\gq}{\mathfrak{q}}
\newcommand{\gr}{\mathfrak{r}}

\newcommand{\0}{\emptyset}
\DeclareMathAlphabet{\mathpzc}{OT1}{pzc}{m}{it}

\DeclareMathAlphabet{\mathpzc}{OT1}{pzc}{m}{it}

\newcommand{\dbra}[1]{% \llrrbracket{..}
  \,[\mkern-6.8mu[\, #1 \,]\mkern-6.8mu]}
\newcommand{\dpar}[1]{% \llrrparen{..}
   (\mkern-4mu (#1 )\mkern-4mu )}

\providecommand\given{}
\newcommand\SetSymbol[1][]{%
\nonscript \: #1 \vert
\allowbreak
\nonscript\:
\mathopen{}}
\DeclarePairedDelimiterX\set[1]\{\}{%
\renewcommand\given{\SetSymbol[\delimsize]}
#1
}
 \DeclarePairedDelimiter\abs{\lvert}{\rvert}

 \newcommand{\usub}[2]{#1_{\textup{#2}}}

 \newcommand{\lan}[3]{\mathcal{L}_{#1 \textup{#2} #3}}

\newcommand{\mdl}[1]{\mathcal{#1}}  % model; e.g. M
\newcommand{\bb}[1]{\mathbb{#1}}

\newcommand{\rest}{\upharpoonright}

\newcommand{\fun}{\longrightarrow}
\newcommand{\efun}{\longmapsto}
\newcommand{\sub}{\subseteq}

\newcommand{\mi}{\smallsetminus}

\newcommand{\la}{\langle}
\newcommand{\ra}{\rangle}

\DeclareMathOperator{\gsk}{\mathbf{K}_+}
\DeclareMathOperator{\ggk}{\mathbf{K}}
\DeclareMathOperator{\ob}{Ob}

\DeclareMathOperator{\fin}{fin}

\DeclareMathOperator{\isp}{I_{sp}}

\DeclareMathOperator{\can}{\mathbf{c}}

\DeclareMathOperator{\TCVF}{TCVF}
\DeclareMathOperator{\TKVF}{TKVF}

\DeclareMathOperator{\vrv}{vrv}
\DeclareMathOperator{\RVH}{RVH}
\DeclareMathOperator{\rv}{rv}
\DeclareMathOperator{\sn}{sn}

\DeclareMathOperator{\csn}{csn}

\DeclareMathOperator{\vv}{val}

\DeclareMathOperator{\VF}{VF}

\DeclareMathOperator{\RV}{RV}

\DeclareMathOperator{\MM}{\mdl {M}}

\DeclareMathOperator{\OO}{\mdl {O}}
\DeclareMathOperator{\res}{res}  % map into the residue field
\DeclareMathOperator{\K}{\Bbbk} %residue field
 \DeclareMathOperator{\ac}{ac}
 \DeclareMathOperator{\rac}{\overline {ac}}
 \DeclareMathOperator{\av}{av}
\newcommand{\dand}{\quad \text{and} \quad}
\newcommand{\LT}{$\lan{T}{}{}$\nobreakdash}

\newcommand{\ddx}{\tfrac{d}{d x}}
\newcommand{\sgn}{\textup{sgn}}

\newcommand{\acpi}{{\ac}{\pi}}
\newcommand{\sig}{\textup{sig}}
\newcommand{\RR}{\R \dpar{ t^{\R} }}
\newcommand{\ORR}{\R \dbra{t^{\R}}}
\newcommand{\gtop}{\mathpzc{0\!\!0}}
\DeclareMathOperator{\KD}{K\Delta}
\newcommand{\Def}{\textup{Def}}
\author[Yimu Yin]{Yimu Yin}
\address{Pasadena \\ California}
\email{yimu.yin@hotmail.com}
\title[\T-convexity with tempered exp]{\T-convex valued fields with tempered exponentiation}

\begin{document}

\begin{abstract}
  We continue the effort of grokking the structure of power-bounded \T-convex valued fields, whose theory is in general referred to as $\TCVF$. In the present paper our focus is on certain expansion of it that is equipped with a tempered exponential function beyond the valuation ring. In order to construct such a tempered exponential function, the signed value group $\Gamma$ is also converted into a model of $T$ plus exponentiation and is in fact identified with (a section of) the residue field via the composition of a diagonal cross-section and an angular component map. In a sense, the resulting universal theory $\TKVF$ is a halfway point between power-bounded $\TCVF$ and exponential $\TCVF$. This theory is reasonably well-behaved. In particular, we show that it admits quantifier elimination in a natural language, a notion of dimension, a generalized Euler characteristic, etc.
\end{abstract}

%\subjclass{03C64, 03C60, 11S80, 03C98, 14B05, 14J17, 32S25, 32S55}
%\keywords{motivic integration, Euler characteristic, \omin-minimal valued field, $T$-convexity, Milnor fiber, topological zeta function}
%\thanks{The research reported in this paper has been fully supported by the ERC Advanced Grant NMNAG}

\maketitle

\tableofcontents

\section{Introduction by way of an example}\label{example}

We have developed a theory of Hrushovski-Kazhdan style motivic integration in power-bounded \T-convex valued fields in \cite{Yin:tcon:I, Yin:tcon:1.5}. For the next step, as already laid out in \cite{hrushovski:kazhdan:integration:vf} (see also \cite{yin:hk:part:3}), one may try to construct the Fourier transform over the additive group. For a comparable theory over the multiplicative group, that is, the Mellin transform, deep obstacles remain insofar as algebraically closed valued fields are concerned. On the other hand, the rich structure in  \T-convex valued fields  lend itself to ways by which such obstacles may be overcome. This is the motivation for the study presented below in this paper. A theory of transforms over both the additive and the multiplicative groups will be the object of a future installment of this series.

The additional structure we shall consider on certain power-bounded \T-convex fields  may seem bizarre at first glance. To have a feel of it, we shall describe its rudiments by way of an example. Some of these are incorporated into an axiomatization in the next section.

A \memph{restricted analytic function} $\R^n \fun \R$ is given on the cube $[-1, 1]^n$ by a power series in $n$ variables over $\R$ that converges in a neighborhood of $[-1, 1]^n$, and $0$ elsewhere. Let $\lan{}{an}{}$ be the language that extends the language of ordered rings with a new function symbol for each restricted analytic function, $\usub{\R}{an}$ the real field with its natural $\lan{}{an}{}$-structure, and  $\usub{T}{an}$ the $\lan{}{an}{}$-theory of $\usub{\R}{an}$. We know from \cite{DMM94} that $\usub{T}{an}$ is a polynomially bounded \omin-minimal theory that is universally axiomatizable and admits quantifier elimination in an enlarged language $\lan{}{an}{}((-)^{-1}, (\sqrt[n]{-})_{n=2,3, \ldots})$ (this language is of course more natural than a brute force definitional extension that achieves the same thing).

A generalized power series with coefficients in the real field $\R$ and exponents in the additive group $\R$ is a formal sum $x = \sum_{q \in \R} a_q t^q$ such that its support $\supp(x) = \set{q \in \R \given a_q \neq 0}$ is well-ordered. Let $\RR$ be the set of all such series. Addition and multiplication on $\R \dpar{ t^{\R} }$ are defined in the expected way, and this makes $\RR$ a field, generally referred to as a Hahn field (with coefficients and exponents both in $\R$). We consider $K \coloneqq \R$ as a subfield of $\RR$ via the embedding $a \efun at^0$. The map $\RR \fun \R$ given by $x \efun \min\supp(x)$ is indeed a valuation. Its valuation ring $\ORR$, $\OO$ for short, consists of those series $x$ with $\min\supp(x) \geq 0$ and its maximal ideal $\MM$ of those series $x$ with $\min\supp(x) > 0$. Its residue field $\K$ admits a section onto $K$, that is, a ring homomorphism $\K \fun \OO$ inverse to the residue map $\res : \OO \fun \K$, and hence is isomorphic to $\R$. It is well-known that $(\RR, \OO)$ is a henselian valued field and $\RR$ is real closed. Restricted analytic functions may be naturally interpreted in $\R \dpar{ t^{\R} }$. According to \cite[Corollary~2.11]{DMM94}, with the induced ordering, $\R \dpar{ t^{\R} }$ is an elementary extension of $\usub{\R}{an}$ and hence a model of $\usub{T}{an}$.

Let $T$ be a complete polynomially bounded \omin-mininal \LT-theory extending $\usub{T}{an}$ that is modeled by $\R$, is universally axiomatized, and also admits quantifier elimination (this is dubbed \memph{hypogenous} in \cite{Yin:tcon:approx}). Recall from \cite[\S~2.2]{Yin:tcon:I} the language $\lan{T}{RV}{}$ and the $\lan{T}{RV}{}$-theory $\TCVF$ of \T-convex valued fields. 

We turn $\RR$ into a $\TCVF$-model as follows. Set $\RV = \R \dpar{ t^{\R} }^{\times} / (1 + \MM)$. Let $\rv : \R \dpar{ t^{\R} }^\times \fun \RV$ be the quotient map. The leading term of a generalized power series in $\RR^\times$ is its first term with nonzero coefficient.  So two series $x$, $y$ have the same leading term if and only if $\rv(x) = \rv(y)$ and hence $\RV$ is isomorphic to the subgroup of $\RR^\times$  of leading terms. There exists a natural isomorphism $a_qt^q \efun (q, a_q)$ from this  group of leading terms to the group $\R \oplus \R^\times$, through which we may identify $\RV$ with $\R \oplus \R^\times$. This identification induces an \memph{angular component map}
\[
\ac: \RR^{\times} \to^{\rv} \RV \to^{a_qt^q \efun a_q} K^\times = \R^\times,
\]
which is to say that the map $\rac \coloneqq \res \circ \ac : \RR^\times \fun \K^\times$ is what is commonly known in the literature as a reduced angular component map (``reduced'' as in ``modulo $\MM$''). Also, since $1 + \MM$ is a convex subset of $\RR^\times$, the total ordering on $\RR^\times$ induces a total ordering $\leq$ on $\RV$, and this ordering $\leq$ is the same as the lexicographic ordering on $\R \oplus \R^\times$ via the identification just made.

Let $\R^{+}$ be the multiplicative group of the positive reals and $\RV^{+} = \R \oplus\R^+ $. Observe that $\R^{+}$ is a convex subgroup of $\RV$. If we write $e^{\R} = \set{e^a \given a \in \R}$ and $\R^\times = \pm e^{\R} \coloneqq e^{\R} \cup - e^{\R}$ then there is a natural isomorphism $ (\R \oplus \R^\times) / \R^{+} \eqqcolon \Gamma \fun \R^\times$ that identifies $\RV^{+} / \R^{+} \cong \R$  with $e^\R$ via the map $q \efun e^q$. Adding a new symbol $\gtop$ to $\RV$, now  we can interpret $\RR$ as an $\lan{T}{RV}{}$-structure, with the \memph{signed valuation} given by
\[
\vv : \RR^\times \fun \Gamma :  x \efun \rv(x) = (q, a_q) \efun \sgn(a_q)e^{-q} \in \R^\times \cong \Gamma,
\]
where $\sgn(a_q)$ is the sign of $a_q$. The discussion in \cite[Example~2.10]{Yin:tcon:I} shows that this indeed yields a $\TCVF$-model.

The crucial difference between $\RR$ and the similarly constructed $\TCVF$-model $\R \dpar{ t^{\Q} }$ is that, in $\RR$,  $\Gamma_{\gtop} \coloneqq \Gamma \cup \{\gtop\} \cong \R$ may be expanded to a \T-model, with $\gtop$ serving as zero, that is isomorphic to the \T-model $\K$, or the section $K$ for that matter. Indeed, all this structure may be expressed through maps from $\RR$ into a single target $K$ instead of two different sorts $\Gamma$ and $\K$; this idea is fleshed out in  the next section.

Denote the real exponentiation by  $\exp : \R \fun \R^+$ and its inverse by $\log : \R^+ \fun \R$. Denote the theory of the structure $(\R, \exp, \log)$ by $T_{\exp}$ and its language by $\lan{}{exp}{}$, where $\R$ is considered as a \T-model. Suppose that $T_{\exp}$ is \omin-minimal, admits quantifier elimination, and is universally axiomatizable (such an exponential \omin-minimal theory may be called  \memph{epigenous}). By \cite[Theorem~B]{DriesSpei:2000}, there is an ample supply of such theories.

It is well-known that $\exp : \R \fun \R^+$ cannot be extended to a total function on $\RR$, but it may be extended to $\OO = \ORR$ and  indeed gives a group isomorphism $\OO \fun \UU^+$, where $\UU^+$ is the subgroup of $\UU \coloneqq \OO^\times$ of positive units in $\OO$, as follows. For  $a \in \UU$, we write $a = \ac(a) + b$ and put
\begin{equation}\label{equ:stan}
\exp(a) = \exp(\ac(a))\exp(b),
\end{equation}
where on the righthand side the first factor is given by the global exponentiation in $K \cong \R$ and the second by the restricted exponentiation in $\RR$, the latter of which could be used if $a \in \MM$; this is compatible with  the restricted exponentiation if $a$ is in the interval $[-1, 1]$. The inverse $\log: \UU^+ \fun \OO$ extends the real logarithm.  So each $a \in \OO$ gives rise to a power function
\begin{equation}\label{RR:power}
(-)^a : \UU^+ \fun \UU^+ : b \efun b^a = \exp(a \log b).
\frac{}{}\end{equation}

Now consider the set
\[
\Delta = \set{a t^{- \log \abs a} \given a \in \R^\times } \sub \RR.
\]
The obvious map $\pi : \Gamma \cong \R^\times \fun \Delta$ is a group homomorphism, actually a cross-section of the value group $\Gamma$, which we call a \memph{diagonal cross-section} of $\Gamma$. Moreover, the map 
\[
at^{- \log \abs a} \efun a \in K^\times,
\]
that is, the restricted angular component map $\ac \rest \Delta$, together with $\pi$, induces a group isomorphism $\Gamma \cong K^\times$, which may be extended to a $T_{\exp}$-isomorphism $\Gamma_{\gtop} \cong K$. The composition $\ac \circ \pi : \Gamma_{\gtop} \fun K$ is abbreviated as $\acpi$.

Note that, traditionally, any cross-section $\csn : \Gamma \fun \VF$ induces a reduced angular component map $\rac : \VF \fun  \K$, but with $\rac(\csn(\Gamma)) = 1$, so the  map $\acpi$ we are considering here  needs to be given separately.

For $\gamma \in \Gamma_{\gtop}$ and $a \in \VF^+$, set
\begin{equation}\label{equ:dequan}
a^\gamma = \pi(\vv(a)^\gamma) \Bigl( \frac{a}{\pi(\vv(a))} \Bigr)^{\acpi(\gamma)} \in \VF^+,
\end{equation}
where the first power is taken via the real exponentiation in $\Gamma_{\gtop}$ and the second one is as given in (\ref{RR:power}). It is straightforward to verify the identities
\[
\vv(a^\gamma) = \vv(a)^\gamma, \quad
  \ac(\pi(\alpha)^\gamma) = \acpi(\alpha)^{\acpi(\gamma)},  \quad (ab)^\gamma = a^\gamma b^\gamma, \quad (a^{\gamma})^{\gamma'} = a^{\gamma \gamma'}.
\]
This operation (\ref{equ:dequan}) is referred to as a \memph{tempered power function} and is denoted by $\varpi_\gamma : \VF^+ \fun \VF^+ : a \efun a^\gamma$. Since $\vv(a+d) = \vv(a)$ if $d$ is sufficiently small, it follows that $\varpi_\gamma$ is differentiable at $a$ if the function
\[
 \vv(a)^{\sharp\sharp} \to^{-/\pi(\vv(a))} \UU^+ \to^{(-)^{\acpi(\gamma)}} \UU^+
\]
is differentiable at $a$, which it is. So $\varpi_\gamma$ is differentiable everywhere and indeed
\begin{align*}
\ddx \varpi_\gamma(a) & =   \pi(\vv(a)^\gamma) \ddx \Bigl( \frac{a}{\pi(\vv(a))} \Bigr)^{\acpi(\gamma)} \\
    & =  \acpi(\gamma) \frac{\pi(\vv(a)^\gamma)}{\pi(\vv(a))}  \Bigl( \frac{a}{\pi(\vv(a))} \Bigr)^{\acpi(\gamma) - 1} \\
    & = \acpi(\gamma) \pi(\vv(a)^{\gamma-1})  \Bigl( \frac{a}{\pi(\vv(a))} \Bigr)^{\acpi(\gamma-1)} \\
    & = \acpi(\gamma) \varpi_{\gamma - 1}(a),
\end{align*}
where the third line holds because $\acpi$ is an isomorphism of $T_{\exp}$-models.

The function on $\VF$ given by $b \efun \varpi_{\vv(b)}$ is ``idempotent'' in the following sense: if $\vv(b) = \vv(b')$ then $\varpi_{\vv(b+b')} = \varpi_{\vv(b)} = \varpi_{\vv(b')}$. This is essentially a trivial statement since $\vv(b) = \vv(b')$ implies $\vv(b+b') = \vv(b)$ (the valuation is signed and hence the leading terms of $b$, $b'$ cannot cancel out).

For each $a \in \VF^+$ with $a \neq 1$, the function $\varrho_a : \Gamma_{\gtop} \fun \VF^+$ given by $\gamma \efun a^\gamma$ is a group monomorphism from  the additive group of the $T_{\exp}$-model $\Gamma_{\gtop}$ into the multiplicative group $\VF^+$. Moreover, setting $\Delta_a = \varrho_a(\Gamma_{\gtop})$, if $\vv(a) \neq 1$ then $\vv \rest \Delta_a$ is a group isomorphism onto $\Gamma^+$ and if $\vv(a) = 1$ then $\Delta_a = K^+$. We think of $\varrho_a$ as a ``warped'' cross-section of $\Gamma_{\gtop}$, reaching the ``limit'' cases as $\vv(a)$ approaches $1$.

For each $a \in \VF^+$ with $a \neq 1$, the function $\overline \exp_a : \VF \fun \VF^+$ given by $x \efun a^{\vv(x)}$ is called a \memph{tempered exponential function}.
 
Tempered power or exponential functions, although the paper's namesake and central to the theory of Mellin transform we envision, are derived from the exponentiation in $K$, the restricted exponentiation in $\RR$, and the diagonal cross-section and hence are a part of the secondary structure. They will not be further studied in this paper.

\begin{rem}
It is easy to see that not only weak \omin-minimality but also dp-minimality fail in this structure of $\RR$ that expands $\TCVF$. It might be a worthwhile exercise to  gauge the exact combinatorial strength of definable sets thereof because this structure is not combinatorial in nature by design. We shall leave it to the interested theologians.
\end{rem}

This paper is a sequel to \cite{Yin:tcon:I, Yin:tcon:1.5} and hence we shall freely use the notation and terminology therein; reminders will be provided along the way.

We  highlight two results. The first is that there is a universal axiomatization of the first-order structure in $\RR$ described above and it admits quantifier elimination. The second is that the universal additive invariant exists and is indeed a generalized Euler characteristic because it takes values in a ring that is canonically isomorphic to $\Z$. More precisely,  let $\VF_* $ be the category of definable sets in $\RR$ and $\textup{Def}_{K}$ the category of $\lan{}{exp}{}$-definable sets in $K$. We abbreviate $\gsk \textup{Def}_{K}$ as $\dO$. There is a commutative diagram
\[
\bfig
  \Square(0,0)/->`->`->`->/<400>[\gsk \VF_*` \dO \llcorner \dO`\ggk \VF_*`\Z;
  \int_+```\int]
 \efig
\]
where $\dO \llcorner \dO$ is constructed from two copies of $\dO$ conjoined at $\N$,  the two horizontal arrows are canonical isomorphisms, and the two vertical arrows are groupifications.

\section{Quantifier elimination}

Let $\lan{\Gamma}{}{}$ be the classical two-sorted language for valued fields, where the two sorts are denoted by $\VF$ and $\Gamma$. The value group sort $\Gamma$ is written multiplicatively. Let $T$ be a hypogenous \LT-theory as described in \S~\ref{example}; so the theory $T_{\exp}$ is epigenous. For convenience, we also assume that \LT~is a \memph{functional language}, see \cite[Remark~2.3]{Yin:tcon:I}. Replace the language of rings in the $\VF$-sort of $\lan{\Gamma}{}{}$ by $\lan{T}{}{}$ and denote the resulting language by $\lan{T \Gamma}{}{}$. We may view $\lan{T\Gamma}{}{}$ as a sub-language of $\lan{T \! \RV}{}{}$. There is an $\lan{T \Gamma}{}{}$-theory such that every model of it can be expanded in a unique way (up to isomorphism) to a model of $\TCVF$ (for this we can ignore the constant symbol $\imath$ in $\lan{T \! \RV}{}{}$). We shall view this theory as an ``$\lan{T \Gamma}{}{}$-reduct'' of $\TCVF$ and, for simplicity, denote it by $\TCVF$ as well. The valuation $\vv : \VF \fun \Gamma_{\gtop}$ is then signed with the middle element $\gtop$.

\begin{defn}
The language $\lan{T\Gamma}{}{}^\flat$ is an expansion of $\lan{T \Gamma}{}{}$ with the following extra symbols:
\begin{itemize}
  \item a unary predicate $K$ in the $\VF$-sort,
  \item a function symbol $\ac$ in  $\VF$-sort,
  \item a function symbol $\pi : \Gamma_{\gtop} \fun \VF$,
  \item a copy of $\lan{T}{}{}$ in the $\Gamma$-sort (merged with the original language of order groups),
  \item two function symbols $\exp$, $\log$ in the $\Gamma$-sort,
  \item two function symbols $\exp$, $\log$ in the $\VF$-sort.
\end{itemize}
\end{defn}

For simplicity, we shall not distinguish in notation the last two pairs of function symbols.

\begin{defn}\label{axioms}
The theory $\TCVF^{\flat}$ is an $\lan{T \Gamma}{}{}^\flat$-expansion of $\TCVF$, which states, in addition to (the $\lan{T\Gamma}{}{}$-equivalents of) the axioms in \cite[Definition~1.3]{Yin:tcon:I}, the following:
\begin{enumerate}[({Ax$^\flat$.} 1)]
  \item $K$ is a subfield of $\OO$.
  \item\label{ax:ang} The function $\ac : \VF^\times \fun K^\times$ augmented by $\ac(0) = 0$ is an angular component map, that is, a group homomorphism such that, for all $a, b \in \OO^\times$,
      \begin{itemize}
        \item $\res(\ac(a)) = \res(a)$,
        \item $\ac(a) = a$ if $a \in K^\times$,
        \item $\ac(a) = \ac(b)$ if $\res(a) = \res(b)$.
      \end{itemize}
  \item The function $\pi : \Gamma \fun \VF^\times$ augmented by $\pi(\gtop) = 0$ is a cross-section, that is, a (multiplicative) group homomorphism with $\vv \circ \pi = \id$.
  \item The functions $\exp, \log : \Gamma_{\gtop} \fun \Gamma_{\gtop}$ are so interpreted that $\Gamma_{\gtop}$ becomes a $T_{\exp}$-model.
  \item For convenience, the functions $\exp, \log : \VF \fun \VF$ are regarded as defined on $K$ since their  interpretations outside $K$  are trivial (constantly zero), and as such they turn $K$ into a $T_{\exp}$-model as well.
  \item\label{ax:diag} The function $\ac \circ \pi \eqqcolon \acpi : \Gamma_{\gtop} \fun K$ is an $\lan{}{exp}{}$-isomorphism.
\end{enumerate}
\end{defn}

\begin{rem}
Here the first two axioms are so formulated that they are in essence universal statements and together imply that the residue map $\res : K \fun \K$ is a field isomorphism. Moreover, since $K$ is a field with nontrivial multiplicative group, the presence of such a field isomorphism as $\acpi$ implies that $\VF$ is nontrivially valued; in $\TCVF$ this is guaranteed by an existential axiom. On the other hand, the presence of a cross-section $\pi$ guarantees that $\vv$ is surjective. The upshot of all this is that  $\TCVF^{\flat}$ would be a universal theory if, in (Ax$^\flat$.~\ref{ax:diag}), $\acpi$ is only required to be an $\lan{}{exp}{}$-monomorphism.
\end{rem}

The signed valuation $\vv$ is by  definition compatible with the total ordering in $\VF$, that is, it is a homomorphism of ordered groups or simply $\vv(\VF^+) = \Gamma^+$, see \cite[Definition~2.7]{Yin:tcon:I}. On the other hand, it would be possible without (Ax$^\flat$.~\ref{ax:diag}) that the angular component map $\ac$ is twisted, that is, $\ac(\VF^+ \mi \OO^\times) = K^-$, despite the restriction $\ac \rest \OO^\times$ being required to be a homomorphism of ordered groups.

Insofar as definable sets are concerned, occasional technical convenience aside, we gain nothing by introducing the $\RV$-sort or even a  residue field sort $\K$ since we already have a cross-section $\pi$, which induces an identification $\RV \cong \Gamma \oplus \K^+$, and $\K$ is now in effect identified with (some structure on) the $T_{\exp}$-model $\Gamma_{\gtop} \cong K$. We will still work with $\RV$ and $\K$ as definable sorts.

In light of (Ax$^\flat$. \ref{ax:diag}), we refer to $\pi$ as a \memph{diagonal cross-section}.

\begin{rem}
As defined in (\ref{equ:stan}), there is a partial exponential function $\exp : \OO \fun \UU^+$ in the $\VF$-sort. If it is only an $\lan{}{exp}{}$-isomorphism $\Gamma_{\gtop} \fun K$ we seek then it seems unnecessary to introduce the diagonal cross-section $\pi$ or even the angular component map $\ac$. But we need a standard part map for the sake of (\ref{equ:stan}). On the other hand, the tempered powers (\ref{equ:dequan}) call for exponentiation in $\Gamma_{\gtop}$ and a cross-section, although not necessarily a diagonal one. A cross-section together with a standard part of course determine an angular component map as well as a diagonal cross-section (if there is one to be had). Thus, in the end, it is leaner to have this $\pi$ to start with.
\end{rem}

\begin{thm}\label{qe}
$\TCVF^{\flat}$ admits quantifier elimination.
\end{thm}
\begin{proof}
Let $\bb M \models \TCVF^{\flat}$, $\bb S \sub \bb M$, and $\bb U \models \TCVF^{\flat}$ be highly saturated. Suppose that there is an $\lan{T \Gamma}{}{}^\flat$-embedding $\sigma : \bb S \fun \bb U$. All we need to do is to pass the Shoenfield test, that is, to  extend $\sigma$ to an $\lan{T \Gamma}{}{}^\flat$-embedding $\bb M \fun \bb U$.

To that end,  consider any element $a \in K(\bb S) \mi \acpi(\Gamma(\bb S))$. Let $\alpha = \acpi^{-1}(a)$ and $\la \Gamma(\bb S), \alpha \ra_{\exp}$ be the $T_{\exp}$-submodel of $\Gamma(\bb M)$ generated by $\Gamma(\bb S)$ and $\alpha$. Then $\acpi (\la \Gamma(\bb S), \alpha \ra_{\exp})$ is the $T_{\exp}$-submodel $\la \acpi(\Gamma(\bb S)), a \ra_{\exp}$ of $K(\bb S)$. We may tentatively  extend $\sigma$ by a $T_{\exp}$-embedding $\la \Gamma(\bb S), \alpha \ra_{\exp} \fun \Gamma(\bb U)$, sending $\alpha$ to $\acpi^{-1}(\sigma(a))$. Let $\la \bb S, \pi(\alpha) \ra_{T \Gamma}$ be the $\TCVF$-submodel of $\bb M$ generated by $\bb S$ and $\pi(\alpha)$; so $\VF(\la \bb S, \pi(\alpha) \ra_{T \Gamma})$ is just the \T-submodel of $\bb M$ generated by $\bb S$ and $\pi(\alpha)$. By the Wilkie inequality \cite[\S~5]{Dries:tcon:97}, we have that $K(\la \bb S, \pi(\alpha) \ra_{T \Gamma}) = K(\bb S)$ and $\Gamma(\la \bb S, \pi(\alpha) \ra_{T \Gamma})$ is the subgroup of $\la \Gamma(\bb S), \alpha \ra_{\exp}$ generated by $\alpha$ over $\Gamma(\bb S)$. The latter equality implies that $\pi(\Gamma(\la \bb S, \pi(\alpha) \ra_{T \Gamma}))$ is the subgroup of $\pi(\la \Gamma(\bb S), \alpha \ra_{\exp})$ generated by $\pi(\alpha)$ over $\pi(\Gamma(\bb S))$ and hence is indeed contained in $\la \bb S, \pi(\alpha) \ra_{T \Gamma}$, also  $\acpi(\Gamma(\la \bb S, \pi(\alpha) \ra_{T \Gamma}))$ is the subgroup of $K(\bb S)$ generated by $a$ over $\acpi(\Gamma(\bb S))$. The proof of quantifier elimination in $T_{\textup{convex}}$  \cite[Theorem~3.10]{DriesLew95} then gives an $\lan{T \Gamma}{}{}$-extension $\sigma_\alpha : \la \bb S, \pi(\alpha) \ra_{T \Gamma} \fun \bb U$ of $\sigma$ with
\[
\sigma_\alpha(\pi(\alpha)) = \pi(\acpi^{-1}(\sigma(a))) = \pi(\sigma_\alpha(\alpha)) \quad \text{and hence} \quad \ac(\sigma_\alpha (\pi(\alpha))) = \sigma(a) = \sigma_\alpha (\acpi(\alpha)).
\]
It follows that, for all $\gamma \in \Gamma(\la \bb S, \pi(\alpha) \ra_{T \Gamma})$,
\[
\sigma_\alpha(\pi(\gamma)) = \pi(\sigma_\alpha(\gamma)) \dand \ac(\sigma_\alpha (\pi(\gamma))) = \sigma_\alpha (\acpi(\gamma)).
\]
For any $a \in \VF^\times(\la \bb S, \pi(\alpha) \ra_{T \Gamma})$ with $\vv(a) = \gamma$, we have
\[
\sigma_{\alpha}(\ac(a)) = \sigma_\alpha(\ac(\pi(\gamma))\ac(a / \pi(\gamma))) = \ac(\sigma_\alpha (\pi(\gamma)))\ac(\sigma_\alpha(a / \pi(\gamma))) = \ac(\sigma_\alpha(a)).
\]
So $\sigma_\alpha$ is almost an $\lan{T \Gamma}{}{}^\flat$-extension of $\sigma$, except that $\Gamma(\la \bb S, \pi(\alpha) \ra_{T \Gamma})$ is not a $T_{\exp}$-model (not even a \T-model for that matter). But we can repeat the above construction for any $\alpha' \in \la \Gamma(\bb S), \alpha \ra_{\exp} \mi \Gamma(\la \bb S, \pi(\alpha) \ra_{T \Gamma})$, and so on. Eventually we shall obtain an $\lan{T \Gamma}{}{}$-extension $\sigma' : \bb S' \fun \bb U $ of $\sigma$ with $\Gamma(\bb S') = \la \Gamma(\bb S), \alpha \ra_{\exp}$ and $K(\bb S') = K(\bb S)$, which then must be an $\lan{T \Gamma}{}{}^\flat$-extension as well. Keep going if there is an $a' \in K(\bb S') \mi \acpi(\Gamma(\bb S'))$, we may assume that $K(\bb S') = \acpi(\Gamma(\bb S'))$ and hence $\bb S'$ is a $\TCVF^{\flat}$-model.

Now we may consider any $a \in K(\bb M) \mi K(\bb S')$. By an almost identical construction (quantifier elimination in $T_{\textup{convex}}$ and the Wilkie inequality are applied twice here, once to all $b \in \la K(\bb S'), a \ra_{\exp}$ and then --- in this order --- as above, to all $\gamma \in \acpi^{-1}(\la K(\bb S'), a \ra_{\exp})$), we may assume $K(\bb M) = K(\bb S')$ and hence $\Gamma(\bb M) = \Gamma(\bb S')$. From this point on, any $\lan{T \Gamma}{}{}$-extension of $\sigma$ is an $\lan{T \Gamma}{}{}^\flat$-extension of $\sigma$ and hence we are done by quantifier elimination in $T_{\textup{convex}}$.
\end{proof}

\begin{cor}\label{prime}
$\TCVF^{\flat}$ has a prime model $\bb P$ and hence is complete.
\end{cor}
\begin{proof}
Let $\bb M \models \TCVF^{\flat}$ and $\bb R$ be the prime model of $T_{\exp}$. Let $\bb R_{K}$, $\bb R_{\Gamma}$ be two copies of $\bb R$ inside $K(\bb M)$, $\Gamma(\bb M)$, respectively, with $\acpi(\bb R_{\Gamma}) = \bb R_{K}$. By the construction in the proof of Theorem~\ref{qe}, the $\lan{T \Gamma}{}{}$-structure $\bb P$ generated by $\bb R_{K}$, $\pi(\bb R_{\Gamma})$ in $\bb M$ is a $\TCVF^{\flat}$-model with $K(\bb P) = \bb R_{K}$ and $\Gamma(\bb P) = \bb R_{\Gamma}$, and is unique (up to isomorphism and independent of $\bb M$). So $\bb P$ is the prime $\TCVF^{\flat}$-model. In light of Theorem~\ref{qe}, we conclude that $\TCVF^{\flat}$ is complete.
\end{proof}

Henceforth we may and shall work in a sufficiently saturated $\TCVF^{\flat}$-model $\bb U$.

Next, we study definable sets in $K$. Our goal is to show that $K$ is stably embedded, that is, all definable sets in $K$ are actually $\lan{}{exp}{}(K)$-definable in $K$. By Theorem~\ref{qe}, we only need to examine sets $A$ defined by a conjunction of the form $\phi(X, a) \wedge K(X)$, where $X = (X_1, \ldots, X_m)$ are the free variables ranging in the $\VF$-sort, $K(X)$ is short for $\bigwedge_i K(X_i)$, and $a$ is a tuple of  parameters in $\VF$, which shall be dropped from the notation below.

\begin{lem}\label{Knopi}
Suppose that none of the function symbols $\vv$, $\ac$, $\pi$, $\exp$, $\log$ occurs in $\phi$. Then $A$ is \LT(K)-definable in $K$.
\end{lem}
\begin{proof}
Since $\vv$ does not occur in $\phi$ and $\phi$ only has $\VF$-sort variables, we may assume that $\phi$ is a $\VF$-sort formula. By the Marker-Steinhorn theorem \cite{makr:stein:deftype}, we may  delete any \LT-literals from $\phi$ and thereby assume that $\phi$ is a formula of the form $\bigwedge_i K(t_i(X)) \wedge \bigwedge_j \neg K(t_j(X))$, where $t_i$, $t_j$ are \LT-terms. In fact, without loss of generality, we may as well assume that $\phi$ is an atomic formula $K(t(X))$. Again, the Marker-Steinhorn theorem guarantees that the set defined by the formula $t(X) = Y \wedge K(X, Y)$
is $\lan{T}{}{}(K)$-definable in $K$, of which $A$ is a projection.
\end{proof}

\begin{lem}\label{kexp}
Suppose that none of the function symbols $\vv$, $\ac$, $\pi$ occurs in $\phi$. Then $A$ is $\lan{}{exp}{}(K)$-definable in $K$.
\end{lem}
\begin{proof}
As in the proof of Lemma~\ref{Knopi}, we may assume that $\phi$ is a $\VF$-sort formula. But we cannot appeal to the Marker-Steinhorn theorem directly because the $\lan{}{exp}{}$-reduct of $\bb U$ is not required to model $T_{\exp}$; indeed, this is the whole point of formulating a theory such as $\TCVF^{\flat}$. We introduce a syntactical complexity function $\mu$, assigning a natural number to each quantifier-free formula  (or a term) that indicates the highest nesting level of the function symbols $\exp, \log : K \fun K$, for instance, $\mu(\exp(\exp(X)\log(X))) = 2$; for this purpose there is no need to distinguish between $\exp$ and $\log$. This is a standard syntactical device and we omit its definition here, see \cite[Definition~2.16]{Yin:int:expan:acvf} for an analogue.

We proceed by induction on $\mu(\phi)$. The base case $\mu(\phi) = 0$ is Lemma~\ref{Knopi}. For the inductive step, consider terms of the form $\exp(t(X))$ or $\log(t(X))$ that occur in $\phi(X)$, where $t(X)$ is an \LT-term. The inductive hypothesis may be applied to the formula $\phi(X) \wedge K(X) \wedge \bigwedge_t \neg K(t(X))$, because all occurrences of $\exp(t(X))$ and $\log(t(X))$ therein  can be replaced by the constant $0$. Thus, without loss of generality, we may assume that $\phi$ contains $\bigwedge_t  K(t(X))$ as a conjunct.  Let $\phi'(X, Z)$ be the formula obtained from $\phi$ by replacing occurrences of $\exp(t(X))$ and $\log(t(X))$ with new variables $Z_{\exp, t}$ and $Z_{\log, t}$. So $\mu(\phi') < \mu(\phi)$. Let $\psi_{\exp}(X, Y, Z)$ be the formula
\[
\bigwedge_t Z_{\exp, t} = \exp(Y_{\exp, t}) \wedge t(X) = Y_{\exp, t};
\]
similarly for $\psi_{\log}(X, Y, Z)$. The inductive hypothesis implies that the set defined by the formula
\[
\phi'(X, Z)  \wedge \psi_{\exp}(X, Y, Z) \wedge \psi_{\log}(X, Y, Z) \wedge  K(X, Y, Z)
\]
is $\lan{}{exp}{}(K)$-definable in $K$, of which $A$ is a projection.
\end{proof}

\begin{prop}\label{Kstab}
In general, $A$ is $\lan{}{exp}{}(K)$-definable in $K$.
\end{prop}
\begin{proof}
Here let $\mu$ be defined so to indicate the highest nesting level of  the  function symbols  $\vv$, $\ac$ in a formula or a term. It will become clear that there is no need to involve $\pi$ in $\mu$. For instance, if $\mu(\phi) = 0$ then none of the $\VF$-sort variables in $\phi$ can appear in the scope of an occurrence of $\pi$, for otherwise $\vv$ must also occur therein and hence $\mu(\phi) > 0$, so all occurrences of $\pi$ in $\phi$, if there is any, may be replaced by parameters in $\VF$.

Consider terms of the form $\vv(t_i(X))$ or $\ac(t_i(X))$ in $\phi$ with $\mu(t_i) = 0$. So we may assume that each $t_i$ is a $\VF$-sort term. Let $t_{ij}(X)$ enumerate the terms of the form $\exp(s(X))$ or $\log(s(X))$ that occur in $t_i$, but not in the scope of any occurrence of  $\exp$ or $\log$. By replacing each $t_{ij}$ with a new $K$-variable as in the proof of Lemma~\ref{kexp}, we may assume that every $t_i$ is an \LT-term (the original set would be a projection). From this point on we proceed by induction on the pair $(\mu(\phi), l)$ of natural numbers, where $l$ is the number of variables that are in  the scope of some occurrence of $\vv$ or $\ac$ (so $l \leq m$). The cases $(0, l)$ have been shown in Lemma~\ref{kexp}.

In the inductive step, for ease of notation, let us assume that every variable $X_j$ indeed occurs in  every term $t_i$ (so $l = m$) and, by the inductive hypothesis, $\phi$ contains $\bigwedge_j X_j \neq 0$ as a conjunct. By \cite[Lemma~3.28]{Yin:tcon:I}, there exists an $\lan{T}{RV}{}$-definable set $V \sub \RV^m$ with $\dim_{\RV}(V) < m$ (see \cite[Definition~2.27]{Yin:tcon:I}) such that, away from $\rv^{-1}(V)$, the functions $\VF^m \fun \VF$ given by the terms $t_i$ are $\rv$-contractible (see \cite[Definition~3.27]{Yin:tcon:I}). Let $B = \rv^{-1}((\K^\times)^m \mi V)$, $C = \rv^{-1}((\K^\times)^m \cap V)$, and $t_{i\downarrow}: (\K^\times)^m \mi V \fun \RV$ be the $\rv$-contraction of $t_i \rest B$.

By \cite[Theorem~A]{Dries:tcon:97}, the set $(\K^\times)^m \cap V$ is \LT($\K$)-definable in $\K$. It follows that $\ac(C)$ is \LT($K$)-definable in $K$ of \omin-minimal dimension less than $m$ and hence may be covered by the graphs of finitely many \LT($K$)-definable functions $K^{m-1} \fun K$. The proof of \cite[Corollary~2.15]{DMM94} shows that every definable function in a hypogenous theory is given piecewise by terms. Therefore, in defining $A \cap \ac(C)$, we can modify $\phi$ so to use less variables, which in effect lowers the number $l$.

On the other hand, by \cite[Proposition~5.8]{Dries:tcon:97}, $\bigcup_i \vrv(t_{i\downarrow}((\K^\times)^m \mi V))$ is finite; we may as well assume it is a singleton, say $\alpha$. Thus, in defining $A \cap \ac(B)$, we can modify $\phi$ by replacing each $\vv(t_i(X))$ with an element of the form $\vv(s_i)$, where $s_i$ is a closed \LT-term (no variables). This also means that all occurrences of $\pi$ below the next nesting level may be eliminated. By \cite[Theorem~A]{Dries:tcon:97} again, the $\rv$-contraction $t'_{i\downarrow}$ of the function $t'_i : B \fun \OO^\times$ given by
\[
x \efun t_i(x)\acpi(\alpha) / \pi(\alpha)
\]
is \LT($\K$)-definable in $\K$ and hence may be regarded as an \LT($K$)-definable function $\ac(B) \fun K$ given by the \LT-term $t'_{i\downarrow}$. For all $b \in \ac(B)$,
\[
t'_{i\downarrow}(b) = \ac(t_i(x)\acpi(\alpha) / \pi(\alpha)) = \ac(t_i(b)).
\]
Thus, in defining $A \cap \ac(B)$, we can modify $\phi$ by replacing each $\ac(t_i(X))$ with  $t'_{i\downarrow}(X)$. All this in effect lowers the number $\mu(\phi)$.

The proposition now follows from the inductive hypothesis.
\end{proof}

It may be argued that creating two separate regions $\Gamma$ and $K$ in the structure only to be immediately identified via an isomorphism is a clumsy setup. Therefore, we introduce a slimmer one-sorted language which will perhaps make our discussion below conceptually  more appealing.

\begin{defn}
The language $\lan{T}{K}{}$ is an expansion of the language $\lan{T}{}{}$. It has just one sort $\VF$ with one extra unary predicate $K$ and five extra function symbols $\vv$, $\ac$, $\pi$, $\exp$, $\log$.
\end{defn}

%Let $\lan{T}{K}{}$ be the sub-language of $\lan{T}{K}{}^\flat$ without the symbols $\ac$, $\pi$, $\exp$, and $\log$. Notwithstanding $K$ being a predicate instead of a sort, we shall treat $\lan{T}{K}{}$ and $\lan{T \Gamma}{}{}$ as the same language.

\begin{defn}
The axioms in Definition~\ref{axioms} can be simplified somewhat in $\lan{T}{K}{}$:
\begin{enumerate}[({Ax.} 1.)]
  \item $K$ is a $T_{\exp}$-model.
  \item $\vv : \VF^\times \fun K^\times$ is a signed valuation map with $K$ also a subfield of $\OO$.
  \item $\ac : \VF^\times \fun K^\times$ is an angular component map.
  \item $\pi : K^\times \fun \VF^\times$ is a section of both $\vv$ and $\ac$, that is,  a group homomorphism such that $\vv \circ \pi = \ac \circ \pi = \id$.
\end{enumerate}
We denote this theory by $\TKVF$.
\end{defn}

Since $K$ and $\Gamma$ are now one and the same, $\TKVF$ is a universal theory. So every substructure of a $\TKVF$-model $\bb M$ is a $\TKVF$-submodel, in particular, every definable closure $\dcl(A)$ of a set $A$ in $\bb M$, is a $\TKVF$-submodel.

\begin{rem}
The relations between the three functions $\ac$, $\vv$, and $\pi$, as group homomorphisms, may be illustrated in a diagram:
\[
\pi(K^\times) \cdot (1 + \MM) \mon^f \VF^\times \two^{\ac}_{\vv} K^\times \epi^g 1
\]
where $f$ is the equalizer and $g$ the coequalizer (in the category of abelian groups, say).
\end{rem}

\begin{thm}\label{kqe}
$\TKVF$ is complete, has a prime model, and admits quantifier elimination.
\end{thm}
\begin{proof}
Given any $\TKVF$-model $\bb M$ there is  a $\TCVF^{\flat}$-model $\bb M^\flat$ such that if we interpret the function $\vv$ in $\bb M$  as the function $\acpi \circ \vv$ in $\bb M^\flat$ then $\bb M^\flat$ admits a $0$-interpretation of $\bb M$ via the identity map on the $\VF$-sort. Moreover, if $\bb N$ is a $\TKVF$-submodel of $\bb M$ then $\bb N^\flat$ is a $\TCVF^{\flat}$-submodel of $\bb M^\flat$. This shows that the corresponding $\TKVF$-model that the prime $\TCVF^{\flat}$-model interprets is indeed the prime $\TKVF$-model. Since a substructure of $\bb M$ is a $\TKVF$-model and $\TCVF^{\flat}$ admits quantifier elimination, it follows that $\TKVF$  passes the Shoenfield test and hence admits quantifier elimination as well. So the prime $\TKVF$-model is an elementary submodel of any $\TKVF$-model, which implies that $\TKVF$ is complete.
\end{proof}

\begin{rem}
Henceforth we shall also view the $\TCVF^{\flat}$-model $\bb U$, via the interpretation just described above, as a sufficiently saturated $\TKVF$-model. Then Proposition~\ref{Kstab} also holds in $\TKVF$. We remark that it appears to be more difficult, perhaps only for psychological reasons, to establish this directly without a detour through $\TCVF^{\flat}$, which is the primary reason for introducing it. On the other hand, the following result is not available for $\TCVF^{\flat}$.
\end{rem}

\begin{cor}\label{fun:yerm}
Every definable function is given piecewise by terms. Consequently, the substructure generated by a set $A$ is indeed $\dcl(A)$.
\end{cor}
\begin{proof}
This holds for any theory that admits quantifier elimination and a universal axiomatization, for instance, hypogenous and epigenous theories, a fact we have used  in the proof of Proposition~\ref{Kstab} above. For a complete argument, see  the proof of \cite[Corollary~2.15]{DMM94}.
\end{proof}

\begin{nota}\label{bi:cor}
For each $a \in \VF$, write $\ac(a)$ as $\dot a$ and $\vv(a)$ as $\ddot a$, and denote the map $a \efun (\dot a, \ddot a)$ by $\av : \VF \fun K^2$. For each set $A$, let $\dot A = \set{\dot a \given a \in A}$ and $\ddot A = \set{\ddot a \given a \in A}$. The definable closure of $\dot A \cup \ddot A$ in the $T_{\exp}$-model $K$ is denoted by $\dcl_K(A)$, which is a $T_{\exp}$-submodel of $K$. Write $\pi(\dcl_K(A))$ as $\Delta(A)$ and, in particular, $\pi(K)$ as $\Delta$. Denote the set $\set{ab \given a \in \dcl_K(A), b \in \Delta(A)}$ by $\KD(A)$, or simply $\KD$ if $\dcl_K(A) = K$. By the construction in the proof of Theorem~\ref{qe},
\[
K(\dcl(A)) = \dcl_K(A) \dand \KD(\dcl(A)) = \KD(A).
\]

Denote $K^\times \times K^\times$ by $\Lambda$ and $\Lambda \cup 0$ by $\Lambda_0$. So the map $\av$ restricts to a bijection between $\KD$ and $\Lambda_0$, and we may think of $(\dot a, \ddot a)$ as the ``binary coordinates'' of $a \in \KD$.

%The definable closure of $\dcl_K(A) \cup \Delta(A)$ is written as $\dcl_\Delta(A)$. Note that in general $\dcl_\Delta(A)$ is a proper $\TKVF$-submodel of $\dcl(A)$.

Let $\lg : \VF^\times \fun \VF^\times$ be the function given by
$a \efun  \pi(\ddot a) \dot a /\ddot a$, augmented by $\lg(0) = 0$. So $\lg(A) \sub  \KD(A)$ for all $A$, in particular, $\lg(\VF) =  \KD$. Moreover,   $\lg(a) = a$ if and only if $a \in \KD$. Since $\ac(\pi(\ddot a) \dot a /\ddot a) = \ddot a \dot a / \ddot a = \ac(a)$, we have $\rv(a) = \rv(\lg(a))$. So $\lg(a)$ may be identified with $(\dot a, \ddot a)$ via the map $\av$. The unique group homomorphism $\sn : \RV \fun \VF^\times$ satisfying $\lg = \sn \circ \rv$ is a \memph{section of $\RV$} in the sense of \cite{Yin:int:expan:acvf}; here $\RV$ is taken as a definable sort. Then, to study  definable sets in light of this connection, we may follow the discussion in \cite[\S~4]{Yin:int:expan:acvf}.
\end{nota}

\begin{rem}\label{term:pole}
Since \LT~is a functional language, we may assume that every atomic formula occurring  in an $\lan{T}{K}{}$-formula $\phi(X)$ is of the form $F(X) = 0$ or $F(X) \leq 0$ or $K(F(X))$, and refer to $F(X)$ as a \memph{top term} of $\phi$. Indeed, replacing  $K(F(X))$ with $\ac(F(X)) = F(X)$, we may even assume that the predicate $K$ does not occur in $\psi$.
\end{rem}

\begin{lem}\label{fiber:TG}
Let $A \sub \VF^n$ be a definable set. Then there exist a definable function $p : \VF^n \fun \KD^m$ and an $\lan{T\Gamma}{}{}$-formula $\phi(X, Y)$ such that, for each $a \in \KD^m$,
\begin{itemize}
 \item $p^{-1}(a) \sub A$ or $p^{-1}(a) \cap A = \0$,
 \item $p^{-1}(a)$ is defined by the formula $\phi(X, a)$.
\end{itemize}
\end{lem}
\begin{proof}
This is just syntactical manipulation. Let $\mu$ be the complexity function that indicates the highest nesting level of  the five function symbols  $\vv$, $\ac$, $\pi$, $\exp$, $\log$ in a formula or a term. Let $\psi(X)$ be a quantifier-free formula that defines $A$. We assume that the predicate $K$ does not occur in $\psi$ and then proceed by induction on $\mu(\psi)$. The base case $\mu(\psi) = 0$ is rather trivial, since $\psi$ is already an \LT-formula.

For the inductive step, let $\zeta(F_i(X))$ enumerate the terms that occur in $\psi$ with $\zeta$ one of the function symbols $\exp$, $\log$, $\vv$, $\ac$, $\pi$ and $F_i$ an \LT-term; so $\mu(\zeta(F_i(X))) = 1$.  Let  $\psi'(X, Y)$ be the formula obtained from $\psi$ by replacing each $\zeta(F_i(X))$ with a new variable $Y_i$ and $A' \sub \VF^{n + l}$ the set defined by it. So $\mu(\psi') < \mu(\psi)$. By the inductive hypothesis, there are a definable function $p' : \VF^{n + l} \fun \KD^{m'}$ and an $\lan{T\Gamma}{}{}$-formula $\phi'(X, Y, Z)$ such that the claim holds with respect to $A'$. Since, for every $a \in \KD$, $(\lg \circ F_i)^{-1}(a) = (\rv \circ F_i)^{-1}(\rv(a))$, we see that the set $(\lg \circ F_i)^{-1}(a) \sub \VF^n$ is $a$-$\lan{T\Gamma}{}{}$-definable, say, by an $\lan{T\Gamma}{}{}$-formula $\phi_i''(X, a)$. For each $x \in \VF^n$, write the tuple $(\zeta(F_i(x)))_i$ as $y_x$. Let $p : \VF^n \fun \KD^{m' + l}$ be the function given by $x \efun  (p'(x, y_x), y_x)$ and $\phi(X, Y, Z)$ the formula $\phi'(X, Y, Z) \wedge \bigwedge_i \phi_i''(X, Y_i)$. These are as desired.
\end{proof}

Observe from the proof  that the codomain of the function $p$ may be further restricted to $K \cup \Delta$.

\begin{defn}\label{rem:fib}
The function $\bar p \coloneqq \av \circ p : A  \fun K^{2m}$ is referred to as a \memph{$K$-partition} of $A$; note that the $\lan{T\Gamma}{}{}$-formula $\phi(X, Y)$ comes as a part of the data. The set $\bar p(A)$ is bijective to $p(A)$ via the map $\av$ and, by Proposition~\ref{Kstab}, is $\lan{}{exp}{}$-definable.

The \memph{dimension} $\dim(\bar p)$ of  $\bar p$ is the maximum of the fiber dimensions $\dim_{\VF}(\bar p^{-1}(\dot b, \ddot b))$, $b \in p(A)$, where the dimension operator $\dim_{\VF}$ is the one in \cite[Definition~2.25]{Yin:tcon:I}.
\end{defn}

So every definable set admits a $K$-partition.

\begin{lem}
Let $\bar q$ be another $K$-partition of $A$. Then $\dim(\bar p) = \dim(\bar q)$.
\end{lem}
\begin{proof}
Since the $\VF$-dimension of an $\lan{T\Gamma}{}{}$-definable set may be equivalently defined as the maximum of the algebraic dimensions of its elements (see the discussion after \cite[Definition~2.25]{Yin:tcon:I}), the proofs of \cite[Lemmas~4.4, 4.5]{Yin:int:expan:acvf} work in the present context almost verbatim. We shall be brief. Let $\bar r$ be the $K$-partition of $A$ given by $a \efun (\bar p(a), \bar q(a))$. Since $K$ is a small set and cannot alter the $\VF$-dimension of an $\lan{T\Gamma}{}{}$-definable set, we have, for each $b \in \bar p(A)$, $\dim_{\VF}(\bar p^{-1}(b)) = \dim(\bar r \rest \bar p^{-1}(b))$, and hence $\dim(\bar p) = \dim(\bar r)$. Similarly  $\dim(\bar q) = \dim(\bar r)$. The lemma follows.
\end{proof}

\begin{rem}\label{dim:well}
So  the \memph{dimension} $\dim(A)$ of a definable set $A$ may be defined as the dimension of any $K$-partition of $A$. Alternatively, it may also be defined as the smallest number $k$ such that there is a definable injection $A \fun \VF^k \times K^l$ (similar to \cite[Lemma~2.26]{Yin:tcon:I}). In particular, $A$ is definably bijective to a set in $K$ if and only if $\dim(A) = 0$. In that case, in light of Proposition~\ref{Kstab}, we may speak of the \memph{\omin-minimal dimension} of $A$, which is denoted by $\dim_K(A)$.
\end{rem}

\section{Special bijections}

We aim to associate a universal additive invariant with $\TKVF$ by combining the constructions in  \cite{Yin:tcon:I, Yin:int:expan:acvf}. In this section we first fuse together the discussions thereof on special bijections. The modifications are local and minute, but the precision afforded by repetition outweighs the pleasure of being concise.

From here on let $\bb S$ be a fixed small substructure of $\bb U$, which is a $\TKVF$-model and is  regarded as a part of the language; so ``$\0$-definable'' or ``definable'' only means $\bb S$-definable. See the beginning of \cite[\S~3]{Yin:tcon:I} for an explanation for doing so.

For convenience, we will speak of $\lan{T}{RV}{}$-definable sets as if $\bb U$ is an $\lan{T}{RV}{}$-structure, even though $\RV$ is only a definable sort. This makes citing definitions and technical results in  \cite{Yin:tcon:I} less cumbersome. In general this involves the translation procedure
\[
\sn(A) = \set{(a, \sn(t)) : (a, t) \in A} \sub \VF^n \times \KD^m
\]
for $A \sub \VF^n \times \RV^m$, where $\sn$ is the section of $\RV$ described in Notation~\ref{bi:cor}. For instance, a set $A \sub \VF^n$ is an \memph{$\RV$-polydisc} if it is of the form $\sn(\gp)$ for some  $\RV$-polydisc $\gp \sub \VF^m \times \RV^{n-m}$ in the sense of \cite[Defintion~2.32]{Yin:tcon:I}, and an  \memph{$\RV$-pullback} if, for all $t \in \rv(A)$, $\rv^{-1}(t) \cap A$ is an $\RV$-polydisc. Note that, unlike in $\TCVF$-models, a set of the form $\rv^{-1}(t)$ contains more than one $\RV$-polydisc.  Then the definition of the  \memph{$\RV$-hull} $\RVH(A)$ of $A$  has to be modified: it is the intersection of all the $\RV$-pullbacks that contain $A$, which is an $\RV$-pullback. So the $\RV$-hull of an $\RV$-pullback is just itself, in particular, the $\RV$-hull of a set in $\KD$ is just itself.

\begin{conv}\label{conv:can}
Since we no longer have the (actual) $\RV$-sort at our disposal, \cite[Convention~5.1]{Yin:tcon:I} needs to be adjusted. For any set $A$, the \memph{regularization} of $A \sub \VF^n$ is now defined as
\[
\can : A \fun \can(A) = \set{(a, \lg(a)) \given a \in A } \sub \VF^n \times \KD^{n}.
\]
We shall tacitly substitute $\can(A)$ for $A$ in the discussion, and whether this substitution has been performed or not should be clear in the context.
\end{conv}

%A coordinate of a set $A$ that  ranges in $\KD$ is referred to as a \memph{$\KD$-coordinate} of $A$. This is often indicated by writing  $A \sub \VF^n \times \KD^m$, but without the implication that $\pr_i(A)$ is not contained in $\KD$ for any $i \leq n$. Of course a $\KD$-coordinate of $A$ may not be a $\KD$-coordinate in a set that contains $A$.

A special bijection as defined in~\cite[Defintion~5.2]{Yin:tcon:I} shall be referred to as an \memph{$\lan{T}{RV}{}$-definable special bijection}.

\begin{defn}
A definable bijection $T : A \fun A^{\sharp}$ is a \memph{special bijection on $A$ of length $1$} if for each $\RV$-polydisc $\gp$ of the form $\gp' \times a \sub \RVH(A)$ with $\gp' = \rv^{-1}(t)$ for some $t \in \RV$ and $a \in \KD$,   there is a $\KD(\gp)$-$\lan{T}{RV}{}$-definable special bijection  $T_{\gp} : \gp' \times a \fun \gp'^\sharp \times a$ of length at most $1$ (regularization is implemented as prescribed by Convention~\ref{conv:can}) such that
\[
T \rest (\gp \cap A) = T_{\gp} \rest (\gp \cap A);
\]
note that all such special bijections $T_{\gp}$ of length $1$ are required to target the same coordinate in $\gp'$, say, the first one. The set $C \sub \RVH(A)$ that contains exactly those $\RV$-polydiscs $\gp$ such that $T_{\gp}$ is of length $1$ is called the \memph{locus} of $T$. For each $\RV$-polydisc $\gp \sub C$, let $\lambda_{\gp} : \gp_{> 1} \fun \gp_1$ be the focus map of $T_{\gp}$. Then $\lambda = \bigcup_{\gp} \lambda_{\gp} \rest (\gp \cap A)_{> 1}$ is  the \memph{focus map} of $T$.

A \memph{special bijection $T$ on $A$ of length $n$} is a composition of $n$ special bijections $T_i$ of length $1$. Each $T_i$ is referred to as a \memph{component} of $T$.
\end{defn}

It might seem that applying $\lan{T}{RV}{}$-definable special bijections fiberwise over $K$-partitions is a more direct route. However, in general, $K$-partitions cannot be achieved through regularization alone and, consequently, they have no counterpart in the dual notion of a blowup, which is an essential ingredient in the computation of the kernel of the lifting homomorphism (see \S~\ref{sect:uai} below). Therefore, proactively, $K$-partitions or any other similar operations cannot be allowed in special bijections.

\begin{rem}\label{spec:rest:acvf}
Let $T : A \fun A^{\sharp}$ be a special bijection with components $T_i$. Observe that if $A$ is an $\RV$-pullback then $A^{\sharp}$ is an $\RV$-pullback. By definition, each $T_i$ is a restriction of a special bijection $R_i$ on $\RVH(\dom(T_i))$ and hence their composition $R$ is a special bijection $\RVH(A) \fun \RVH(A^\sharp)$. It follows that, for any $\RV$-polydisc $\gp \sub \RVH(A^{\sharp})$ and any $\KD(\gp)$-$\lan{T}{RV}{}$-definable set $B \sub \gp$, the set $R^{-1}(B) \sub \RVH(A)$ is contained in an $\RV$-polydisc and is indeed $\KD(\gp)$-$\lan{T}{RV}{}$-definable, and $R \rest R^{-1}(B)$ is a $\KD(\gp)$-$\lan{T}{RV}{}$-definable special bijection.
\end{rem}

For the remainder of this section, let $\mu$ be the syntactical complexity function for the five function symbols  $\vv$, $\ac$, $\pi$, $\exp$, and $\log$, as described in the proof of Lemma~\ref{fiber:TG}. Let $\zeta$ stand for any one of these five function symbols. For any quantifier-free formula or term $\phi(X)$, if $\zeta(F(X))$ occurs in $\phi$ then it is a  \memph{$\zeta$-term} of $\phi$, and a \memph{top $\zeta$-term}  is a $\zeta$-term that does not occur in the scope of another occurrence of $\zeta$ (which may or may not be a top term of $\phi$).

%, and if $F(X)$ occurs in $\phi$ but not in a term other than itself then $F(X)$ is a  \memph{top term} of $\phi(X)$.

\begin{lem}\label{spec:bi:poly:cons:noRV}
Let $\tau(X): \VF^n \fun \VF$ be a term. Let $\gp \sub \VF^n$ be an $\RV$-polydisc and $R : \gp \fun A$ a special bijection. Then there is a special bijection $T$ on $A$ such that the function $\tau \circ R^{-1} \circ T^{-1}$ is $\rv$-contractible.
\end{lem}
\begin{proof}
First observe that if the assertion holds for one such term $\tau$ then it holds simultaneously for any finite number of them. Let $\zeta(F_{ki}(X))$ enumerate the $\zeta$-terms of $\tau$ such that $\mu(F_{ki}(X)) = k$; so each $F_{0i}(X)$ is just an \LT-term. By compactness, we may  concentrate on a single $\RV$-polydisc $\gp_0 \sub A$. By Remark~\ref{spec:rest:acvf}, $R \rest R^{-1}(\gp_0)$ is  $\KD(\gp_0)$-$\lan{T}{RV}{}$-definable. By \cite[Theorem~5.5]{Yin:tcon:I}, there exists a $\KD(\gp_0)$-$\lan{T}{RV}{}$-definable special bijection $T_0 : \gp_0 \fun \gp^\sharp_0$ such that, for every $\RV$-polydisc $\gq \sub \gp^\sharp_0$, $\rv(B^{\gq}_{0i})$ is a singleton for all $i$, where
\[
B^{\gq}_{0i} =  (F_{0i} \circ R^{-1} \circ T_0^{-1})(\gq) \sub \VF.
\]
If $\zeta$ is $\pi$ or $\exp$ or $\log$ and $\vv(B^{\gq}_{0i}) = \pm 1$ then $\zeta(B^{\gq}_{0i})$ could be the two-element set $\{\zeta(\lg(B^{\gq}_{0i})), 0\}$. By  Remark~\ref{spec:rest:acvf}, $(T_0 \circ R \circ F_{0i}^{-1})(\lg(B^{\gq}_{0i})) \sub \gp^\sharp_0$ is $\KD(\gq)$-$\lan{T}{RV}{}$-definable and hence, by \cite[Corollary~5.6]{Yin:tcon:I} and compactness, we may extend $T_0$ so that either $B^{\gq}_{0i} = \{\lg(B^{\gq}_{0i})\}$ or $\lg(B^{\gq}_{0i}) \notin B^{\gq}_{0i}$. Consequently,  $\zeta(B^{\gq}_{0i})$ is a singleton $\{a^{\gq}_{0i}\} \sub \KD(\gq)$ for all the five possibilities of $\zeta$.

Next, we concentrate on a single $\RV$-polydisc $\gp_1 \sub T_0(\gp_0)$. Let $F^{\gp_1}_{1i}(X)$ be the $\lan{T}{}{}$-term obtained from $F_{1i}(X)$ by replacing each $\zeta(F_{0i}(X))$ with $a^{\gp_1}_{0i}$. As above, there exists a $\KD(\gp_1)$-$\lan{T}{RV}{}$-definable special bijection $T_1$ on $\gp_1$ such that, for every $\RV$-polydisc $\gq \sub T_1(\gp_1)$,
\[
(\zeta \circ F^{\gp_1}_{1i} \circ R^{-1} \circ T_0^{-1} \circ T_1^{-1})(\gq)
\]
is a singleton for all $i$.

Repeating this procedure for all the $\zeta$-terms $F_{ki}(X)$ of higher complexities, we see that there is a special bijection $T$ on $A$ as desired.
\end{proof}

We will have a more powerful analogue of \cite[Theorem~5.5]{Yin:tcon:I} than the lemma above.

\begin{cor}\label{VF:image:RV:inj}
For any definable set $A \sub \VF^n$ there is a special bijection $T : A \fun A^\sharp$ such that $A^\sharp$ is an $\RV$-pullback.
\end{cor}

In particular, if $\dim(A) = 0$ then $A^\sharp$ is a set in $\KD$; of course, as we have pointed out above, $A$ is definably bijective to a set in $K$ or, equivalently, $\KD$ if and only if $\dim(A) = 0$, so the point is just that we can choose a special bijection as a witness.

\begin{proof}
Let $\phi(X)$ be a quantifier-free formula that defines $A$. Let $\tau_i(X)$ enumerate the  top $\zeta$-terms of $\phi$. By Lemma~\ref{spec:bi:poly:cons:noRV}, there exists a special bijection $T : \VF^n \fun C$ such that every function $\tau_i \circ T^{-1}$ is $\rv$-contractible. This means that $(\tau_i \circ T^{-1})(\gp)$ is a singleton $\{a^{\gp}_{i}\} \sub \KD(\gp)$ for every $\RV$-polydisc $\gp \sub C$. Let $\phi^{\gp}(X)$ be the formula obtained from $\phi$ by replacing each $\tau_{i}(X)$ with $a^{\gp}_{i}$ and $A^{\gp}$ the set defined by $\phi^{\gp}(X)$. It follows that
\[
T^{-1}(\gp) \cap A = T^{-1}(\gp) \cap A^{\gp}.
\]
Since we may assume that the predicate $K$ does not occur in $\phi$ (Remark~\ref{term:pole}), $\phi^{\gp}(X)$ is an \LT-formula, which can then be transformed into an $\lan{T}{RV}{}$-formula $\phi^{\gp}_*(X)$ as described in \cite[Convention~3.1]{Yin:tcon:I}, that is, every \LT-term occurs in the scope of an instance of $\rv$. Let $\sigma_j(X)$ enumerate the top \LT-terms of $\phi^{\gp}_*(X)$. Applying Lemma~\ref{spec:bi:poly:cons:noRV} again, we may assume that every function $\sigma_j \circ T^{-1}$ is $\rv$-contractible as well. Thus, either $T^{-1}(\gp) \sub A^{\gp}$ or $T^{-1}(\gp) \cap A^{\gp} = \0$. The claim follows.
\end{proof}

\begin{lem}\label{two:spec:ann}
Let $T$, $R$ be special bijections on an $\RV$-disc $\gp \sub \VF$. Let $R_i$ enumerate the components of $R$, $\hat R_i = R_i \circ \cdots \circ R_1$, $\lambda_i$ the focus map of $R_i$, and $A_i \sub  \gp$ the image of (the graph of) $\lambda_i$ under $\hat R^{-1}_i$. If $T(\gp) \mi T ( A_i )$ is an $\RV$-pullback for every $i$ then $R \circ T^{-1}$ is $\rv$-contractible.
\end{lem}

Since $T(\gp)$ is itself an $\RV$-pullback, the  extra clause in the conclusion just means that  $T ( A_i )$ is also an $\RV$-pullback and indeed must be a set in $\KD$ because it is of dimension  $0$. It is not enough to require that  $T ( A_i )$ be a set in $\KD$ since, in that case, it could still be a proper subset of an $\RV$-pullback contained in $T(\gp)$. In contrast, in the $\TCVF$-model $\bb U$ different $\RV$-polydiscs are automatically pairwise disjoint.

\begin{proof}
Suppose for contradiction that $R \circ T^{- 1}$ is not $\rv$-contractible. Then there are an $\RV$-polydisc $\gq \sub  T (\gp)$ and an $i$ such that $(\rv \circ \hat R_i \circ T^{- 1})(\gq)$ is a singleton but $(\rv \circ \hat R_{i+1} \circ T^{- 1})(\gq)$ is not. Since $(\hat R_{i} \circ T^{- 1})(\gq)$ is an open polydisc of the form $\ga \times a \sub \VF \times \KD^m$, we have $\lambda_{i + 1} \cap (\hat R_i \circ T^{- 1})(\gq) \neq \0$. So $T ( A_{i + 1} ) \sub \gq$,  contradiction.
\end{proof}

Here is a minor upgrade of Lemma~\ref{spec:bi:poly:cons:noRV}:

\begin{lem}\label{spec:bi:any}
If $\tau: \VF^n \fun \VF$ is a definable function instead of a term then the conclusion of Lemma~\ref{spec:bi:poly:cons:noRV} still holds.
\end{lem}
\begin{proof}
By Corollary~\ref{fun:yerm}, the function $\tau \rest \gp$ is given piecewise by terms, say, $\tau_i(X) \rest B_i = \tau \rest B_i$, where $(B_i)_i$ is a partition of $\gp$. By Remark~\ref{spec:rest:acvf} and Corollary~\ref{VF:image:RV:inj}, there is a special bijection $T$ on $A$ such that each $(T \circ R)(B_i)$ is an $\RV$-pullback. Then,  by Lemma~\ref{spec:bi:poly:cons:noRV}, we may assume that every function $\tau_i \circ R^{-1} \circ T^{-1}$ is already $\rv$-contractible. The claim follows.
\end{proof}

We are now ready to state a better generalization of \cite[Theorem~5.5]{Yin:tcon:I}:

\begin{thm}\label{spec:bi:term:cons:disc}
Let $\tau : \VF^n \times \KD^m \fun \VF$ be a definable function, $\gp \sub \VF^n$ an $\RV$-polydisc, and $R : \gp \fun A$ a special bijection. Then there is a special bijection $T$ on $A$ such that, for every $c \in \KD^m$, the function $\tau(-,  c) \circ R^{-1}  \circ T^{-1}$ is $\rv$-contractible.
\end{thm}

Note that the difference between this statement and Lemma~\ref{spec:bi:any} is that the special bijection $R$ here does not depend on the parameters in  $\KD^m$, otherwise it would simply be a consequence of Lemma~\ref{spec:bi:any} and compactness.

%Without loss of generality, we may assume that each $F_{0i}$ is an \LT-term that does not involve the variables $Y$ (or we may simply stipulate $\mu(Y) = 1$).

\begin{proof}
As Lemma~\ref{spec:bi:any}, if the assertion holds for one such function $\tau$ then it holds simultaneously for any finite number of such functions.

We do induction on $n$. For the base case $n = 1$, by Lemma~\ref{spec:bi:any}, for each $c \in \KD^m$  there is a $\KD(c)$-definable special bijection $T_{c}$ on $A$ such that, for every $\RV$-polydisc $\gq \sub T_{c}(A)$ and every $i$, the set $(\rv \circ f_c \circ T_{c}^{-1})(\gq)$ is a singleton $\set{t_{c,\gq}}$. Let $\lambda_{c, k}$ enumerate (the graphs of) the focus maps of the components of $T_{c}$. Let $\lambda : \bigcup_{c} \biguplus_k \lambda_{c, k} \times c \fun A$ be the definable function induced by the bijections $T^{-1}_{c}$. By Corollary~\ref{VF:image:RV:inj}, there is a special bijection $T$ on $A$ such that $T(\ran(\lambda))$ is an $\RV$-pullback (of dimension $0$). So, by Lemma~\ref{two:spec:ann}, every function $T_{c} \circ T^{-1}$ is $\rv$-contractible. It follows that for every $\RV$-polydisc $\gr \sub T(A)$ and every $c \in \KD^m$ there is an $\RV$-polydisc $\gq_{\gr} \sub T_{c}(A)$ such that $T^{-1}(\gr) \sub T_{c}^{-1}(\gq_{\gr})$ and hence the set $(\rv \circ f_c \circ T^{-1})(\gr)$ is the singleton $\set{t_{c, \gq_{\gr}}}$. This concludes the base case of the induction.

For the inductive step, without loss of generality, we may concentrate on a single $\RV$-polydisc $\gp \sub A$ of the form $\rv^{-1}(t) \sub \VF^n$. By the inductive hypothesis, for each $a \in \gp_1$, there exists a $\KD(a)$-definable special bijection $R_{a}$ on $\gp_{> 1}$ such that, for every $c \in \KD^m$, the function $\tau(a, -, c) \circ  R^{-1} \circ R_a^{-1}$ is $\rv$-contractible. Thus,
\[
G_{a,c} = \lg \circ \tau(a, -, c) \circ  R^{-1} \circ R_a^{-1} \circ \lg^{-1}
\]
may be understood as a function between two sets in $\KD$. Let $U_{a, k}$ enumerate the loci of the components of $R_{a}$ and $\lambda_{a, k}$ the corresponding focus maps. By compactness, there are
\begin{itemize}
  \item a quantifier-free formula $\psi$ such that $\psi(a, c)$ defines (the graph of) $G_{a,c}$,
  \item a quantifier-free formula $\theta$ such that $\theta(a)$ determines the sequence $(\lg(U_{a, k}))_k$ as well as the coordinates targeted by $\lambda_{a, k}$.
\end{itemize}
Recall Remark~\ref{term:pole}.  Let $H_{j}(X_1, \ldots)$ enumerate the top terms of $\psi$ and $\theta$. Applying the inductive hypothesis again, we obtain a special bijection $T_1$ on $\gp_1$ such that, for every tuple $b$ in $\KD$ of the right length,  the function $H_{j}(-, b) \circ T_1^{-1}$ is $\rv$-contractible. This means that, for every $\RV$-polydisc $\gq \sub T_1(\gp_1)$ and every $a_1, a_2 \in T_1^{-1}(\gq)$,
\begin{itemize}
  \item the formulas $\psi(a_1, c)$, $\psi(a_2, c)$ define the same function for every $c \in \KD^{m}$,
  \item the special bijections $R_{a_1}$, $R_{a_2}$ may be glued together in the obvious way to form one special bijection on $\{a_1, a_2\} \times \gp_{> 1}$.
\end{itemize}
Consequently, $T_1$ and $R_{a}$ induce a special bijection $T$ on $\gp$ such that $\tau(-,  c) \circ R^{-1} \circ T^{-1}$ is $\rv$-contractible for every $c \in \KD^{m}$, as desired.
\end{proof}

\begin{lem}\label{simul:special:dim:1}
Let $A$, $B$ be definable sets of dimension $\leq 1$ and $f : A \fun B$ a definable function.  Then there exist special bijections $T_A : A \fun A^{\sharp}$, $T_B : B \fun B^{\sharp}$ and a bijection $f^{\sharp}_{\downarrow}: \lg(A^{\sharp}) \fun \lg(B^{\sharp})$ such that $A^{\sharp}$, $B^{\sharp}$ are $\RV$-pullbacks and the diagram (\ref{fub:one}) commutes.
\begin{equation}\label{fub:one}
 \bfig
  \square(0,0)/->`->`->`->/<600,350>[A`A^{\sharp}`B`B^{\sharp};
  T_A`f``T_B]
 \square(600,0)/->`->`->`->/<600,350>[A^{\sharp}`\lg(A^{\sharp})`B^{\sharp} `\lg(B^{\sharp});  \lg`f^{\sharp}`f^{\sharp}_{\downarrow}`\lg]
 \efig
\end{equation}
\end{lem}
\begin{proof}
By Corollary~\ref{VF:image:RV:inj} and Lemma~\ref{spec:bi:any},  we may assume that $A$, $B$ are already $\RV$-pullbacks and $f$ is already $\rv$-contractible. Let $p$ be a $K$-partition of (the graph of) $f$. For each $d \in \ran(p)$, set $\dom(f_{d}) = A_{d}$ and $\ran(f_{d}) = B_{d}$. Then each $A_{d}$ is contained in a set of the form $\sn(\VF \times \RV^m)$, similarly for $B_{d}$. Then, by~\cite[Lemma~2.24]{Yin:tcon:I}, there is a definable finite partition of $A_{d}$ into sets $A_{d,i}$ such that each $f_{d} \rest A_{d,i}$ has the so-called disc-to-disc property (\cite[Definition~2.23]{Yin:tcon:I}), that is, for every open polydisc $\ga  \sub \dom(f_d)$, $f_d(\ga)$ is also an open polydisc, and vice versa. Applying  Lemma~\ref{spec:bi:any} to the function on $A$ given by $a \efun (d, i)$ for $a \in A_{d,i}$, where $(d, i)$ is a tuple in $K$, we may assume that each $A_{d,i}$ is an $\RV$-pullback and hence $f$ has the disc-to-disc property, in particular, for each $\RV$-polydisc $\gp \sub A$, $f(\gp)$ is an open polydisc contained in an $\RV$-polydisc. By Lemma~\ref{spec:bi:any} again, there is a special bijection $T_B: B \fun B^{\sharp}$ such that $(T_B \circ f)^{-1}$ is $\rv$-contractible. The adjustments we have done so far are summarized in the opening sentence of the proof of~\cite[Lemma~5.10]{Yin:tcon:I} and, consequently, from this point on, we may simply follow the argument therein.
\end{proof}

\section{Universal additive invariant}\label{sect:uai}

One advantage of having a diagonal cross-section in our formalism is that, unlike in, say, \cite{hrushovski:kazhdan:integration:vf, Yin:int:acvf, Yin:tcon:I, Yin:int:expan:acvf}, specialization of the universal additive invariant to $\Z$ (generalized Euler characteristic) does not require a lengthy structural analysis
of definable sets in the $\RV$-sort, since all the relevant prerequisites are now fullfilled by the \omin-minimality of $K$.

\begin{defn}[$\VF$-categories]\label{defn:c:VF:cat}
The objects of the category $\VF_n$ are the definable sets of dimension at most $n$. Its morphisms are precisely the definable bijections between such objects. Set $\VF_* = \bigcup_n \VF_n$.
\end{defn}

By Remark~\ref{dim:well}, $\VF_0$ is equivalent to the core $\Def^C_K$ of the category $\textup{Def}_K$ of $K(\bb S)$-definable sets in $K$ and hence, by Proposition~\ref{Kstab}, is equivalent to the category of definable sets in any \omin-minimal fields. So the semiring $\gsk \Def^C_K$ may be denoted by $\dO$ for brevity.

\begin{thm}\label{sig:dete}
For any two objects $A, B \in \Def^C_K$, $[A] = [B]$ in $\gsk \Def^C_K$ if and only if $\dim_K(A) = \dim_K(B)$ and $\chi(A) = \chi(B)$.
\end{thm}
Here $\chi$ is the Euler characteristic associated with \omin-minimal (field) structures (see \cite[\S~4.2]{dries:1998}).
\begin{proof}
See \cite[\S~8.2.11]{dries:1998}.
\end{proof}

Therefore, $\gsk \Def^C_K$ admits an explicit description as follows. Its underlying set is $(\{0\} \times \N) \cup (\N^+ \times \Z)$ and, for $(a, b), (c, d) \in \dO$,
\begin{equation}\label{oring}
(a, b) + (c, d) = (\max\{a, c\}, b+d), \quad (a, b) \cdot (c, d) = (a + c, b  d).
\end{equation}
The dimensional part is lost in the groupification $\ggk \Def^C_K$, that is, $\ggk \Def^C_K \cong \Z$. The sub-semiring over $\{0\} \times \N$ is isomorphic to the monoid $\N$ and hence is simply written as such.

\begin{defn}[$\Lambda$-categories]\label{def:Ga:cat}
For each $n \in \N$, the category $\Lambda[n]$ is a copy of $\Def^C_K$. Set $\Lambda[\leq \! n] = \coprod_{i \leq n} \Lambda[i]$ and $\Lambda[*] = \coprod_n \Lambda[n]$.
\end{defn}

\begin{rem}
For a definable set $U$ in $K$, we write $[U]_n$ to indicate that this is an element in $\gsk \Lambda[n]$. An object $U \in \Lambda[\leq \! n]$ is in general understood as a coproduct $\coprod_{i \leq n} U_i$ with $U_n \neq \0$, and to emphasize this we shall write $U \in^* \Lambda[\leq \! n]$ or $[U] \in^* \gsk \Lambda[\leq \! n]$.
\end{rem}

So $\gsk \Lambda[*] \cong \dO[X]$ and $\ggk \Lambda[*] \cong \Z[X]$.

\begin{defn}\label{def:L}
The \memph{$n$th canonical lifting map} $\mathbb{L}_n: \ob \Lambda[n] \fun \ob \VF_n$ is given by
\begin{equation}\label{normpull}
\mathbb{L}_n(U) = \underbrace{(1 + \MM) \times \ldots \times (1 + \MM)}_{\text{$n$-fold}} \times U.
\end{equation}
Set $\mathbb{L}_{\leq n} = \coprod_{i \leq n} \mathbb{L}_{i}$ and $\mathbb{L} = \bigcup_n \bb L_{\leq n}$.
\end{defn}

\begin{prop}\label{L:sur:c}
The lifting map $\bb L$ induces a surjective homomorphism, also denoted by $\bb L$,
\[
\gsk \Lambda[*] \fun \gsk \VF_*.
\]
\end{prop}
\begin{proof}
It is clear that every isomorphism class of $\Lambda[\leq \! n] $ is mapped via $\bb L$ into an isomorphism class of $\VF_n$. On the other hand, by Corollary~\ref{VF:image:RV:inj}, every isomorphism class of $\VF_n$ contains an $\RV$-pullback which, due to the presence of the map $\av \circ \lg: \VF^\times \fun \Lambda$ (Notation~\ref{bi:cor}), may be assumed to take the form indicated in (\ref{normpull}). This implies surjectivity.
\end{proof}

Denote the subset $K^+ \times (0, 1) \sub \Lambda$ by $Q^+$, similarly $K^- \times (-1, 0) \sub \Lambda$ by $Q^-$, $Q^+ \cup Q^- \sub \Lambda$ by $Q$, and $Q \cup 0 \sub \Lambda_0$ by $Q_0$.

\begin{defn}\label{defn:blowup:coa}
Suppose that $n > 0$ and $U \in \Lambda[n]$. The \memph{elementary blowup} of $U$ is the object
\[
U^{\sharp} = U \amalg (U \times Q) \in \Lambda[n{-}1] \amalg \Lambda[n].
\]
Let $V \in \Lambda[n]$, $C \sub V$, and $F : U \fun C$ be an $\Lambda[n]$-morphism. Then $U^{\sharp} \uplus (V \mi C)$ is a \memph{blowup of $V$ via $F$}, written as $V^{\sharp}_F$; the subscript $F$ is dropped  if there is no danger of confusion. The object $C$ is called the \memph{locus} of the blowup $V^{\sharp}_F$. A \memph{blowup of length $n$} is a composition of $n$ blowups.

If $n = 0$ then the only blowup of an object is the identity, corresponding to the empty locus.
\end{defn}

\begin{defn}
For any object $U = \coprod_{i \leq n} U_i \in \Lambda[\leq \! n]$, its \memph{signature} $\sig(U)$ is the sequence of pairs of integers $(\dim_K(U_i), \chi(U_i))_i$. By Theorem~\ref{sig:dete} again, $[U] \in \gsk \Lambda[*]$ is completely determined by $\sig(U)$, so we may write $\sig([U]) = (\dim_K([U_i]), \chi([U_i]))_i$. The sum
\[
\sum_i (-1)^i \chi([U_i])
\]
is an \memph{Euler characteristic} of $[U]$ and is denoted by $\chi([U])$.

\end{defn}

Observe that, since $\chi(Q) = 2$, if $U^\sharp$ is a blowup of $U$ then
\begin{equation}\label{eublow}
\chi([U^\sharp]) = \chi([U]).
\end{equation}

\begin{lem}\label{evenup}
Let $n > 0$ and $U \in^* \Lambda[\leq \! n]$. For each $i \leq n$ let $m_i$ be an integer so that $\sum_i (-1)^i m_i = 0$. Let $l$ be a natural number with
\[
l \geq \max\{\dim_K(U_i) + 2: i \leq n\}.
\]
Then there is a blowup $U^{\sharp}$ of $U$ with
\[
\sig(U^{\sharp}) = ((l-2, m_0 + \chi(U_0)), (l, m_1 + \chi(U_1)), \ldots, (l, m_n + \chi(U_n))).
\]
\end{lem}
\begin{proof}
We first assume that $l$ is even. Since $U_n$ is not empty, we may blow up one point in it and thereby assume $\dim_K(U_n) \geq 2$ and $U_{n-1} \neq \0$ as well. Let $A$, $B_i$, and $C_j$ be finitely many pairwise disjoint cells contained in $U_n$ with $\dim_K(A) = \dim_K(B_i) = 0$ and $\dim_K(C_j) = 1$ such that
\begin{equation}\label{blowup:sum}
\sum_{k = 0}^{l/2 - 1} 2^k\chi(A) + \sum_i \chi(B_i) + \sum_j \chi(C_j) = m_n.
\end{equation}
Now we blow up $U_n$ at $A$, then at $A \times Q$, $A \times Q \times Q$, and so on, altogether for $l/2$ times, and then at $\bigcup_{i,j} B_i \cup C_j$. The end result is an object $V_{n-1} \amalg V_n$ with
\begin{align*}
\chi(V_n) & =  \chi(A)(\chi(Q)^{l/2} -1) + \chi \Bigl(\bigcup_{i,j} B_i \cup C_j \Bigr) (\chi(Q) - 1) + \chi(U_n) \\
    & =  \sum_{k = 0}^{l/2 - 1} 2^k\chi(A) + \sum_i \chi(B_i) + \sum_j \chi(C_j) + \chi(U_n)\\
    & = m_n + \chi(U_n)
\end{align*}
and, by (\ref{blowup:sum}) directly, $\chi(V_{n-1}) = m_n$. So we have
\[
\sig(V_n) = (l, m_n + \chi(U_n)) \dand \sig(U_{n-1} \uplus V_{n-1}) = ( l-2, \chi(U_{n-1}) + m_n).
\]
Replacing $m_{n-1}$ with $m_{n-1} - m_n$, we see that the assertion follows from an induction on $n$.

If $l$ is odd then we take $A$ to be a cell of dimension $1$ and change the first summand in (\ref{blowup:sum}) to $\sum_{k = 0}^{(l-1)/2 - 1} 2^k\chi(A)$. The rest of the modification is straightforward.
\end{proof}

\begin{lem}\label{blowup:equi:class:coa}
For any two objects $U, V \in^* \Lambda[\leq \! n]$ with $n > 0$, there are isomorphic blowups $U^{\sharp}$, $V^{\sharp}$ of $U$, $V$ if and only if $\chi([U]) = \chi([V])$.
\end{lem}

Observe that if $U \in^* \Lambda[\leq \! n]$ and $V \notin^* \Lambda[\leq \! n]$ then it is not possible to have isomorphic blowups of $U$, $V$.

\begin{proof}
The ``only if'' direction is clear by (\ref{eublow}). For the ``if'' direction, let $l \in \N$ be sufficiently large and $m_i = \chi([V]_i) -\chi([U]_i)$. By Lemma~\ref{evenup}, there is a blowup $U^{\sharp}$ of $U$ with the signature in question; similarly for $V$ but with $m_i = 0$ for all $i$. Thus, by Theorem~\ref{sig:dete}, $[U^{\sharp}] = [V^{\sharp}]$.
\end{proof}

\begin{defn}
Let $\isp[n]$ be the full subcategory of $\Lambda[\leq \! n] \times \Lambda[\leq \! n]$ of pairs $(U, V)$ such that there exist isomorphic blowups $U^{\sharp}$, $V^{\sharp}$. Let $\isp[*] = \bigcup_{n} \isp[n]$.
\end{defn}

We will just write $\isp$ for all these subcategories if there is no danger of confusion.

\begin{rem}\label{semi:concr}
By Lemma~\ref{blowup:equi:class:coa}, $\isp$  may also be understood as a binary relation on isomorphism classes, which may be equivalently defined by the condition: for $[U], [V] \in^* \gsk \Lambda[\leq \! n]$ with $n > 0$,
\[
([U], [V]) \in \isp[n] \quad \text{if and only if} \quad \chi([U]) = \chi([V]);
\]
note that $\isp[0]$ is trivial, since the only blowups are the identities. Then it is clear that $\isp[n]$ is a semigroup congruence relation and $\isp[*]$ is a semiring congruence relation.

Let $\Lambda^{\fin}[0]$ be the full subcategory of finite objects of $\Lambda[0]$. Replacing $\Lambda[0]$ with $\Lambda^{\fin}[0]$ in $\Lambda[*]$, we obtain $\Lambda'[*]$. If $U \in^* \Lambda[\leq \! n]$ and $V \in^* \Lambda[\leq \! m]$ with $n \leq m$ then
\begin{gather*}
  [U] + [V] \in^* \gsk \Lambda[\leq \! m], \quad \chi([U] + [V]) = \chi([U]) + \chi([V]),\\
  [U][V] \in^* \gsk \Lambda[\leq \! m{+}n], \quad \chi([U][V]) = \chi([U])\chi([V]),
\end{gather*}
which are just the operations described in (\ref{oring}). Thus, we have canonical isomorphisms
\[
\gsk \Lambda[0] / \isp = \gsk \Lambda[0] \cong \dO \cong \gsk \Lambda'[*] / \isp.
\]
Denote by $\dO \llcorner \dO$ the semiring whose underlying set is two copies of $\dO$ conjoined at $\{0\} \times \N$, with the second one in the notation referred to as the dominator, and whose operations are given by those in $\dO$ for $(a, b)$, $(c, d)$ in the same copy of $\dO$ and, otherwise, with  $(c, d)$ in the dominator, by
\[
 (a, b) + (c, d) = (c, b+d), \quad (a, b) \cdot (c, d) = (c, b  d).
\]
So $\gsk \Lambda[*] / \isp \cong \dO \llcorner \dO$. In the groupification of $\dO \llcorner \dO$, both copies of $\dO$ are collapsed to $\Z$ and they must coincide because they are conjoined at $\N$. So  $\ggk \Lambda[*] / \isp \cong \Z$.
\end{rem}

\begin{lem}\label{blowup:leng:ma}
Let $U, V \in^* \Lambda[\leq \! k]$ with $[U] = [V]$. Let $U_1$, $V_1$ be blowups
of $U$, $V$ of lengths $m$, $n$, respectively. Then there are blowups $U_2$, $V_2$ of $U_1$, $V_1$ of lengths $n$,
$m$, respectively, with $[U_2] = [V_2]$.
\end{lem}
\begin{proof}
This is the same statement as \cite[Lemma~5.29]{Yin:tcon:I}, whose formal proof works here almost verbatim.
\end{proof}

Let $U \in^* \Lambda[\leq \! n]$ and $T$ be a special bijection on $\bb L U$. The set $(\av \circ \lg \circ T)(\mathbb{L} U)$ is regarded as an object in $\Lambda[\leq \! n]$ and is denoted by $U_{T}$.

\begin{lem}\label{special:to:blowup:coa}
The object $U_T \in^* \Lambda[\leq \! n]$ is isomorphic to a blowup of $U$.
\end{lem}
\begin{proof}
By induction on the length $l$ of $T$ and Lemma~\ref{blowup:leng:ma}, this is immediately reduced to the case $l = 1$, which is rather clear.
\end{proof}

\begin{defn}\label{rela:unary}
Let $A \sub A' \sub \VF^{n}$ and $B \sub B' \sub \VF^{n'}$ be definable sets. Suppose that, for some $k \leq \min\{n, n'\}$, $A \sub \VF^k \times \KD^{n-k}$ and $B \sub \VF^k \times \KD^{n'-k}$. For $i \leq k$ let $\tilde{i} = \{1, \ldots, k\} \mi \{i\}$.  We say that a definable bijection $G : A' \fun B'$ with $G(A) = B$ is \memph{relatively unary in the $(i, j)$-coordinate} if $G(x) = x$ for $x \in A' \mi A$ and $(\pr_{\tilde{j}} \circ G)(x) = \pr_{\tilde{i}}(x)$ for  $x \in A$. If, in addition, $G \rest A_a$ is also a special bijection for every $a \in \pr_{\tilde{i}} (A)$ then $G$ is \memph{relatively special} in the $(i, j)$-coordinate.
\end{defn}

Alternatively, for simplicity, we may take $i = j$ in this definition and stipulate that the condition in question holds up to coordinate permutation.

\begin{exam}
Every  special  bijection $T$ of length $1$ is clearly relatively special, but not vice versa (to begin with, the relevant data is not necessarily uniform across the fibers in each $\RV$-polydisc).
\end{exam}

\begin{lem}\label{unary:piece}
Let $A \sub \VF^k \times \KD^{m}$, $B \sub \VF^k \times \KD^{m'}$, and $f : A \fun B$ be a definable bijection. Then there is a definable finite partition $A_i$ of $A$ such that each $f \rest A_i$ is a composition of relatively unary bijections.
\end{lem}
\begin{proof}
Let $p$ be a $K$-partition of (the graph of) $f$. Then we may view each $p^{-1}(a)$ as an $a$-$\lan{T}{RV}{}$-definable bijection between a subset of $\VF^k \times \RV^{m}$ and a subset of $\VF^k \times \RV^{m'}$. Since the bijection on $A$ given by $a \efun (a, p(a, f(a)))$ is relatively unary in all the relevant coordinates, we see that the claim  follows from \cite[Lemma~5.24]{Yin:tcon:I} and compactness.
\end{proof}

Observe that if a  bijection  is piecewise (over a definable finite partition) a composition of relatively unary bijections then it is indeed a composition of relatively unary bijections.

Let $A \sub \VF^{n}$ be a definable set and $T_i : A \fun A^\sharp$ a definable bijection that is relatively special in the $(i, i)$-coordinate (up to coordinate permutation). Suppose that the ``$\KD$-coordinates'' of $A$ are passed on to $A^\sharp$ via $T_i$ in the obvious way (the simplest example is that $T_i$ is the regularization of $A$). Moreover, for all $a \in A_{\tilde i}$ (here $\tilde{i} = \{1, \ldots, n\} \mi \{i\}$), the image $T_i(A_{a}) = A_{a}^\sharp$ of the fiber $A_{a} \sub A$ over $a$ is an $\RV$-pullback over $a$.  Let
\[
 A_{-i} = \bigcup_{a \in A_{\tilde i}} a \times \lg(A_{a}^\sharp) \sub \VF^{n-1} \times \KD^{m_i},
\]
which  we think of as obtained from $A^\sharp$ by ``$\lg$-contracting'' the coordinate in question. Note that, under Convention~\ref{conv:can}, $A^\sharp$ can be fully recovered from  $A_{-i}$. Let $\hat T_i : A \fun A_{-i}$ be the function induced by $T_i$.

For any suitable $j \leq n-1$, we repeat the above setup on $A_{-i}$ with respect to the $j$-coordinate, and thereby obtain a definable set $A_{-j} \sub \VF^{n-2} \times \RV^{m_j}$ and a function $\hat T_{j} : A_{-i} \fun A_{-j}$ that is induced by a relatively special bijection $T_j : A^\sharp \fun A^{\sharp\sharp}$ (the coordinate in $A^\sharp$ that gets ``$\lg$-contracted'' in $A_{-i}$ is simply carried along by $T_j$). Continuing thus,  a sequence of such bijections  $T_{\sigma(1)}, \ldots, T_{\sigma(n)}$ and a corresponding function $\hat T_{\sigma} : A \fun \KD^{l}$ result, where $\sigma$ is the permutation of $\{1, \ldots, n\}$ in question. The composition $T_{\sigma(n)} \circ \ldots \circ T_{\sigma(1)}$, which is referred to as the \memph{lift} of $\hat T_{\sigma}$, is denoted by $T_{\sigma}$. Note that $T^{-1}_{\sigma}$ is $\rv$-contractible. Also, since $T_{\sigma}(A)$ is an $\RV$-pullback, the set $\hat T_{\sigma}(A)$ may be regarded as an object $\coprod_{i \leq n} \hat T_{\sigma}(A)_i \in \Lambda[\leq \! n]$ such that $(b, a) \in \hat T_{\sigma}(A)_i$ if and only if there is an $\RV$-polydisc $\gp \times a \sub T_{\sigma}(A)$ with $\lg(\gp) = b$.

\begin{defn}\label{defn:standard:contraction}
The definable function $\hat T_{\sigma}$, or the object $\hat T_{\sigma}(A) \in \Lambda[\leq \! n]$, is called a \memph{standard contraction} of  $A$.
\end{defn}

By Corollary~\ref{VF:image:RV:inj}, there are abundant standard contractions in stock for any definable set.

\begin{lem}\label{kernel:dim:1:coa}
Suppose that $[A] = [B]$ in $\gsk \VF_1$ and $U, V \in \Lambda[\leq \! 1]$ are two standard contractions of them. Then $(U, V) \in \isp$.
\end{lem}
\begin{proof}
By Lemma~\ref{simul:special:dim:1}, there are special bijections $T$, $R$ on $\bb L U$, $\bb L V$ such that $U_{T}$, $V_{R}$ are isomorphic. So the assertion follows from Lemma~\ref{special:to:blowup:coa}.
\end{proof}

\begin{lem}\label{blowup:same:RV:coa}
Let $U^{\sharp}$ be a blowup of $U \in \Lambda[\leq \! n]$ of length $l$. Then $\bb L U^{\sharp}$ and $\bb L U$ are isomorphic.
\end{lem}
\begin{proof}
By induction this is immediately reduced to the case $l=1$. Then, in the presence of $\lg$, a bijection between $\bb L U^{\sharp}$ and $\bb L U$ may be easily constructed.
\end{proof}

\begin{lem}\label{isp:VF:fiberwise:contract:isp:coa}
Let $A', A'' \sub \VF^n \times \KD^m$ be definable sets with $\pr_{\leq n} (A') = \pr_{\leq n} (A'') = A$ and $\dim(A) = l$. Suppose that, for some $k$ and every $a \in A$, $(A'_a, A''_a) \in \isp[k]$. Let $\hat T_{\sigma}$, $\hat R_{\sigma}$ be standard contractions of $A'$, $A''$. Then
\[
(\hat T_{\sigma}(A'), \hat R_{\sigma}(A'')) \in \isp[k{+}l].
\]
\end{lem}

Note that the condition $(A'_a, A''_a) \in \isp[k]$ is construed over the substructure $\bb S \la a \ra$.

\begin{proof}
By induction on $n$ this is immediately reduced to the case $n=1$. Let $f$ be the definable function on $A$ given by $a \efun (\sig([A'_a]), \sig([A''_a]))$ (in fact, by compactness, the range of $f$ is finite). Let $\phi'$, $\phi''$ be quantifier-free formulas that define $A'$, $A''$.  Applying Theorem~\ref{spec:bi:term:cons:disc} to $f$ and the top terms of $\phi'$, $\phi''$ (for the first clause below Lemma~\ref{spec:bi:any} suffices, but for the second clause the stronger Theorem~\ref{spec:bi:term:cons:disc} is needed), we obtain a special bijection $F: A \fun A^{\sharp}$ such that $A^{\sharp}$ is an $\RV$-pullback and, for all $\RV$-polydisc $\gp \sub A^{\sharp}$,
\begin{itemize}
 \item $(f \circ F^{-1})(\gp)$ is a singleton,
 \item for all $a_1, a_2 \in F^{-1}(\gp)$, $A'_{a_1} = A'_{a_2}$ and $A''_{a_1} = A''_{a_2}$.
\end{itemize}
Let $B' = \bigcup_{a \in A} F(a) \times A'_a$, similarly for $B''$. Note that $B'$, $B''$ are obtained via special bijections on $A'$, $A''$. It follows that, for all $e \in \pr_{> 1} (A')$, $B'_e$ is an $\RV$-pullback that is $e$-definably bijective to the $\RV$-pullback $T_{\sigma}(A')_e$. By Lemma~\ref{kernel:dim:1:coa}, we have
\[
(\lg(B'_e), \hat T_{\sigma}(A')_e) \in \isp[1]
\]
and hence, by compactness, $(\lg(B'), \hat T_{\sigma}(A')) \in \isp[k{+}1]$. Of course the same holds for $\lg(B'')$ and $\hat R_{\sigma}(A'')$. On the other hand, for all $\RV$-polydisc $\gp \sub A^{\sharp}$, the assumption gives
\[
(\lg(B')_{\lg(\gp)}, \lg(B'')_{\lg(\gp)}) \in \isp[k]
\]
and hence $(\lg(B'), \lg(B'')) \in \isp[k{+}1]$. Since $\isp$ is a congruence relation, the lemma follows.
\end{proof}

\begin{cor}\label{contraction:same:perm:isp:coa}
Let $A', A'' \sub \VF^n$ and suppose that there is a definable bijection between them that is relatively unary in the $(i, i)$-coordinate for some $i \leq n$. Then we have, for any permutation $\sigma$ of $\{1, \ldots, n\}$ with $\sigma(1) = i$ and any standard contractions $\hat T_{\sigma}$, $\hat R_{\sigma}$ of $A'$, $A''$,
\[
(\hat T_{\sigma}(A'), \hat R_{\sigma}(A'')) \in \isp.
\]
\end{cor}
\begin{proof}
This is immediate by Lemmas \ref{kernel:dim:1:coa} and \ref{isp:VF:fiberwise:contract:isp:coa}.
\end{proof}

\begin{lem}\label{subset:partitioned:2:unit:contracted}
Let $12$, $21$ denote the permutations of $\{1, 2\}$. Let $A \sub \VF^2 \times \KD^m$ be a definable set. Then there are a definable injection $f : A \fun \VF^2 \times \KD^l$  that  is relatively unary in both coordinates and standard contractions $\hat T_{12}$, $\hat R_{21}$ of $f(A)$ with $[\hat T_{12}(f(A))] = [\hat R_{21}(f(A))]$ in $\gsk \Lambda[\leq \! 2]$.
\end{lem}
\begin{proof}
Let $p$ be a $K$-partition of $A$. Since the bijection on $A$ given by $a \efun (a, p(a))$ is relatively unary in both coordinates,  the assertion simply follows from \cite[Lemma~5.26]{Yin:tcon:I} and compactness.
\end{proof}

\begin{lem}\label{contraction:perm:pair:isp:coa}
Let $A \sub \VF^n$ be a definable set. Let $i, j \in \{1, \ldots, n\}$ be distinct indices and $\sigma_1$, $\sigma_2$ two permutations of $\{1, \ldots, n\}$ such that
\[
\sigma_1(1) = \sigma_2(2) = i, \quad \sigma_1(2) = \sigma_2(1) = j, \quad \sigma_1
\rest \set{3, \ldots, n} = \sigma_2 \rest \set{3, \ldots, n}.
\]
Then, for any standard contractions $\hat T_{\sigma_1}$, $\hat T_{\sigma_2}$ of $A$,
\[
(\hat T_{\sigma_1}(A), \hat T_{\sigma_2}(A)) \in \isp.
\]
\end{lem}
\begin{proof}
This is the same statement as \cite[Lemma~5.38]{Yin:tcon:I},  the ingredients whose proof formally depends on have all been reproduced above, namely Lemmas~\ref{isp:VF:fiberwise:contract:isp:coa}, \ref{subset:partitioned:2:unit:contracted} and Corollary~\ref{contraction:same:perm:isp:coa}. So the argument therein may be quoted here with virtually no changes.
\end{proof}

Indeed, we have reproduced all the results that the proof of \cite[Proposition~5.39]{Yin:tcon:I} formally depends on, namely Proposition~\ref{L:sur:c}, Lemmas~\ref{blowup:same:RV:coa}, \ref{kernel:dim:1:coa}, \ref{unary:piece}, \ref{special:to:blowup:coa}, \ref{contraction:perm:pair:isp:coa}, and Corollaries~\ref{contraction:same:perm:isp:coa}, \ref{VF:image:RV:inj}, so the following crucial description of the kernel of $\bb L$ may be established by the same argument.

\begin{prop}\label{kernel:L:dag:coa}
For $U, V \in \Lambda[\leq \! n]$,
\[
[\bb L U] = [\bb L  V] \quad \text{if and only if} \quad ([ U], [ V]) \in \isp.
\]
\end{prop}

In light of this description, Proposition~\ref{L:sur:c}, and Remark~\ref{semi:concr}, we conclude:

\begin{thm}\label{main:prop:k:vol:dag}
For each $n \geq 0$ there is a canonical isomorphism of Grothendieck semigroups
\[
\int_{+} : \gsk  \VF_n \fun \gsk  \Lambda[\leq \! n] /  \isp
\]
such that
\[
\int_{+} [A] = [U]/  \isp \quad \text{if and only if} \quad  [A] = [\bb L U].
\]
Putting these together, we obtain a canonical isomorphism of Grothendieck semirings
\[
\int_{+} : \gsk \VF_* \fun \gsk  \Lambda[*] /  \isp \cong \dO \llcorner \dO
\]
and, upon groupification, a canonical isomorphism of Grothendieck rings
\[
\int : \ggk \VF_* \fun \ggk  \Lambda[*] /  \isp \cong \Z.
\]
\end{thm}

\providecommand{\bysame}{\leavevmode\hbox to3em{\hrulefill}\thinspace}
\providecommand{\MR}{\relax\ifhmode\unskip\space\fi MR }
\providecommand{\MRhref}[2]{%
  \href{http://www.ams.org/mathscinet-getitem?mr=#1}{#2}
}
\providecommand{\href}[2]{#2}

%%---------------------------------------------------------------------
%%Included for Gather Purpose only:
%%input "C:\localtexmf\bibtex\bib\mybib\MYbib.bib"
%\bibliographystyle{amsplain}
%\bibliography{C:/Users/yimuy/Dropbox/mybib/MYbib}
%%---------------------------------------------------------------------

\end{document}